\documentclass[journal]{IEEEtran}
\IEEEoverridecommandlockouts

\usepackage{url,cite}
\usepackage{url,color,amsmath,graphicx,enumerate,tikz}
\usepackage{amssymb,dsfont,mathrsfs}
\usepackage{thmtools}
\usepackage{soul}

\usepackage[colorlinks]{hyperref} 

\usetikzlibrary{shapes, backgrounds, mindmap, shadows}
\usetikzlibrary{automata,positioning}
\usetikzlibrary{mindmap}
\def\QED{~\rule[-1pt]{5pt}{5pt}\par\medskip}
\newenvironment{proof}{{\bf Proof: \ }}{ \hfill \QED}

\let\ALP \mathcal
\let\FLD \mathscr

\renewcommand{\Re}{\mathbb{R}}
\newcommand{\Na}{\mathbb{N}}
\newcommand{\beqq}[1]{\begin{align*}#1\end{align*}}
\newcommand{\beq}[1]{\begin{align}#1\end{align}}

\newcommand{\ex}[1]{\mathbb{E}\left[#1\right]}
\renewcommand{\sp}{\text{span}}

\newcommand{\pr}[1]{\mathbb{P}\left\{#1\right\}}

\newcommand{\new}[1]{{\color{black}#1}}

\newtheorem{theorem}{\bf Theorem}
\newtheorem{proposition}{\bf Proposition}
\newtheorem{claim}{\bf Claim}
\newtheorem{lemma}[theorem]{\bf Lemma}
\newtheorem{corollary}[theorem]{\bf Corollary}

\newtheorem{assumption}{Assumption}

\newtheorem{remark}{Remark}
\newtheorem{definition}{Definition}

\let\ALP  \mathcal
\let\FLD  \mathscr

\title{Probabilistic Contraction Analysis of Iterated Random Operators}

\author{Abhishek Gupta\thanks{Abhishek Gupta is with the Electrical and Computer Engineering Department at The Ohio State University, Columbus, OH, USA. Email: \texttt{gupta.706@osu.edu}.} \quad Rahul Jain\thanks{Rahul Jain is with the Electrical Engineering Department at the University of Southern California, Los Angeles, CA, USA. Email: \texttt{rahul.jain@usc.edu}} \quad Peter Glynn\thanks{Peter Glynn is with the Management Science and Engineering Department at Stanford University, Stanford, CA, USA. Email: \texttt{glynn@stanford.edu}.\newline  The first author would also like to thank Dr. Hiteshi Sharma and Prof. William Haskell for numerous discussions on the results reported in this paper.}}

\begin{document}
\maketitle
\begin{abstract}
In many branches of engineering, Banach contraction mapping theorem is employed to establish the convergence of certain deterministic algorithms. Randomized versions of these algorithms have been developed that have proved useful in data-driven problems. In a class of randomized algorithms, in each iteration, the contraction map is approximated with an operator that uses independent and identically distributed samples of certain random variables. This leads to iterated random operators acting on an initial point in a complete metric space, and it generates a Markov chain. In this paper, we develop a new stochastic dominance based proof technique, called probabilistic contraction analysis, for establishing the convergence in probability of Markov chains generated by such iterated random operators in certain limiting regime. The methods developed in this paper provides a general framework for understanding convergence of a wide variety of Monte Carlo methods in which contractive property is present. We apply the convergence result to conclude the convergence of fitted value iteration and fitted relative value iteration in continuous state and continuous action Markov decision problems as representative applications of the general framework developed here. 
\end{abstract}

\section{Introduction}\label{sec:intro}
Let $(\ALP X,\rho)$ be a complete metric space with the metric $\rho$. Let $T:\ALP X\rightarrow\ALP X$ be a contraction operator over this space, i.e., there exists an $\alpha\in [0,1)$ such that for any $x,y\in\ALP X$, we have
\[\rho(T(x),T(y))\leq \alpha \rho(x,y).\]
According to the Banach contraction mapping theorem \cite{ali2006}, for any starting point $x_1\in\ALP X$, the sequence generated by $x_{k+1} = T(x_k), k\in\Na$ converges to the unique fixed point $x^\star$ of $T$. Many recursive algorithms in optimization, such as value and policy iteration algorithm, gradient descent algorithms, primal-dual algorithms, etc. can be viewed as iterating a contraction operator on Euclidean spaces or complete function spaces \cite{bertsekas1995dynamic,bertsekas1996neuro,almudevar2014approximate}. 

Often, evaluating $T(x)$ may be computationally challenging or just not possible, particularly when it involves computing expectations of certain functions of random variables. For such cases, Monte-Carlo algorithms have been devised in which random variables are used to approximate the computation of the operator $T$. Thus, iteration $k$ of the algorithm is an application of a random operator $\hat T^n_k$, where $n$ is the number of random samples used to approximate $T$ at the $k^{th}$ iteration and are assumed to be independent of the past samples. Consequently, $\hat T^n_k$ is independent of the past random operators $(\hat T^n_l)_{l=0}^{k-1}$. Examples of such algorithms are \new{adaptive Monte-Carlo algorithms \cite{kollman1999,desai2001adaptive, desai2001markov,ahamed2006adaptive}, sequential Monte-Carlo algorithms \cite{heng2020controlled,boustati2020generalised},} temporal difference methods \cite{tsitsiklis1997analysis}, reinforcement learning algorithms \cite{almudevar2014approximate,bertsekas1996neuro,bertsekas2012q}, approximate dynamic programming for risk sensitive Markov decision problems \cite{yu2018approximate}, first-step analysis, and optimal stopping problems. The methods developed in this paper provide a general framework for understanding convergence of a wide variety of Monte Carlo methods in which contractive property is present.

Intuitively speaking, such independent random operators when applied recursively to any starting point $x_0$ lead to a random sequence $(\hat X^n_k)_{k\in\Na}$ that drifts towards the fixed point $x^\star$ of $T$. The goal of this paper is to devise a novel framework for understanding asymptotic behavior (that is, consistency) of the iterates generated from such iterated random operators. We introduce the notion of a probabilistic fixed point of the random operators, which is adapted from \cite{haskell2016}.
\begin{definition}
A point $\bar x\in\ALP X$ is a \textit{probabilistic fixed point} of $(\hat T^n_k)_{k\in\Na,n\in\Na}$ if for any initial point $x_0\in\ALP X$ and for every $\kappa>0$,  $(\hat X^n_k)_{k,n\in\Na}$ satisfies 
\beqq{\lim_{n\rightarrow\infty}\limsup_{k\rightarrow\infty}\; \pr{\rho(\hat X^n_k, \bar x)>\kappa} = 0. }
\end{definition}
In this paper, we provide sufficient conditions under which $x^\star$ is a probabilistic fixed point of the iterated random operators.

The analysis of iterated random operators was first carried out in \cite{dubins1966invariant,barnsley1985iterated,barnsley1989recurrent}; see also the recent surveys on the topic \cite{diaconis1999iterated,stenflo2012survey,duflo2013random}. Under the assumption that the random operators have a negative Lyapunov exponent and the underlying space is Polish, these papers analyzed the stability of the Markov chain generated by the iterated random operators. The authors developed a novel ``backward iteration'' argument and proved that the Markov chain converges to a unique invariant distribution at a geometric rate. While the rate of convergence could be inferred, it is not helpful when one is interested in sample complexity bounds for recursive stochastic algorithms and compute how far the Markov chain is from the fixed point $x^\star$ of the deterministic contraction operator. 

Accordingly, this paper presents an analysis of iterated random operators with negative Lyapunov exponent satisfying certain consistency properties to complete metric spaces under somewhat different assumptions than in \cite{diaconis1999iterated}. We show that $x^\star$, the fixed point of the deterministic operator, is the probabilistic fixed point of the random operators. Our approach is based on a novel stochastic dominance type argument. We refer to this new approach as the probabilistic fixed point approach to distinguish it from other approaches. We demonstrate applicability of our general framework to empirical value iteration for the discounted-cost and average-cost continuous-state Markov decision problems (MDPs). Let us first consider some learning examples that can be modeled within the iterated random operator framework.

\noindent\textit{Example 1:} Consider the infinite horizon discounted cost Markov decision problem in which $s\in \ALP S$ is the system state, $a \in \ALP A$ is the control action, $c(s,a)$ is the one stage cost, $\alpha \in (0,1)$ is the discount factor. The state transition function is $S' = g(S,A,Z)$, where $Z\in\ALP Z$ is the exogenous noise. We assume that the state and the action spaces are finite sets.

Let $\gamma(a|s)$ denote a stationary policy -- which is a measure over $\ALP A$ for every $s\in\ALP S$. Let $\Gamma$ denote the set of all such stationary policies. The goal of the decision maker is to minimize the total discounted cost by solving the following minimization problem:
$$v^*(s)=\min_{\gamma\in\Gamma}\ex{\sum_{k=0}^{\infty}\alpha^kc(S_k,A_k)\Bigg|s_0=s, A_k \sim \gamma(\cdot|S_k)}.$$ 
The optimal value function $v^*$ is a fixed point of a contraction operator $T$, which is defined as
\beq{[T(v)](s) = \min_{a\in\ALP A} \Big\{c(s,a) + \alpha \ex{v(g(s,a,Z))}\Big\}.\label{eqn:Tdiscount}}

The operator $T$ is the \textit{Bellman operator}. It is not difficult to show that $T$ is a contraction operator over the normed vector space $\ALP V:=\{v:\ALP S\rightarrow\Re\}$ endowed with the sup norm. The optimal value function $v^*$ is the limit of the value iteration algorithm
\beqq{ v_{k+1}(s) = [T(v_k)](s) = \min_{a\in\ALP A} \{ c(s,a) + \alpha \ex{v_k(g(s,a,Z))} \}.}
When computing the expectation exactly is difficult or not possible, the Bellman operator is often approximated using the \textit{empirical Bellman operator}, $\hat T^n_k:\ALP V\to\ALP V$  (see, for example, \cite{haskell2016}) defined as
\beq{& \hat v^n_{k+1}(s) = [\hat T^n_k(\hat v^n_k)](s) \nonumber\\
& := \min_{a\in\ALP A} \Bigg\{ c(s,a) + \alpha \frac{1}{n}\sum_{i=1}^n\hat v^n_k(g(s,a,Z_{k,i})) \Bigg\}, \label{eqn:hatTdiscount}}
where $\{Z_{k,i}\}_{i=1}^n$ are independent samples of the noise variable $Z_k$. This algorithm is known as empirical value iteration. Since $\{Z_{k,i}\}_{i=1}^n$ are i.i.d., we conclude that $\hat T^n_k$ is i.i.d. and $(\hat v^n_k)_{k\in\Na}$ is a Markov chain. 
\hfill\QED

\noindent {\it Example 2.} Consider the same setting as above, but instead of the discounted cost MDP, we will consider  minimizing the average cost:
\beqq{\min_{\gamma\in\Gamma}\lim_{K\rightarrow\infty}\mathbb{E}& \Bigg[\frac{1}{K+1}\sum_{k=0}^{K}c(S_k,A_k)\Bigg| s_0=s, A_k \sim \gamma(\cdot|S_k)\Bigg].} 
Under some conditions on the MDP's state transition function, one can show that the optimal solution exists \cite{puterman2014}. In this case, the optimality condition is given by a tuple $(v^*,J^*)$, where $v^*$ is the optimal relative value function and $J^*$  is a real number called the optimal gain. The Bellman operator $T$ satisfies:
\[ v^*(s) + J^* = [T(v^*)](s) := \min_{a\in\ALP A} \Big\{c(s,a) + \ex{v^*(g(s,a,Z))}\Big\}.\]
Under certain conditions (see, for example, \cite[Sec 8.5.2]{puterman2014}), one can show that $T$ is a contraction operator over a quotient space with the span seminorm given by
\beq{span(v) = \max_{s\in\ALP S} v(s) - \min_{s\in\ALP S} v(s).\label{eqn:span}}
The computation of the optimal relative value function $v^*$ is the limit of the following iterative process, which is known as the \textit{relative value iteration} (RVI) algorithm, 
\beq{&v_{k+1}(s) = [T(v_k)](s) \nonumber\\
& := \min_{a\in\ALP A} \{ c(s,a) + \ex{v_k(g(s,a,Z))}  \}- \min_{s\in\ALP S} v_k(s).\label{eqn:Taverage}}
In this case, if the expectation operator is difficult to evaluate, then the \textit{empirical relative value iteration} (ERVI) \cite{gupta2015edp} is defined as 
\beq{& \hat v^n_{k+1}(s) = [\hat T^n_k(\hat v^n_k)](s) \nonumber\\
& := \min_{a\in\ALP A} \Bigg\{ c(s,a) + \frac{1}{n}\sum_{i=1}^n\hat v^n_k(g(s,a,Z_{k,i})) \Bigg\} - \min_{s\in\ALP S} \hat v^n_k(s),\label{eqn:hatTaverage}}
where again, $\{Z_{k,i}\}_{i=1}^n$ is a sequence of independent and identically distributed samples of the noise variable. These noise samples are generated independently from the past samples at every iteration $k$. Once again, we observe that $(\hat T^n_k)_{k\in\Na}$ is an i.i.d. sequence of operators and $(\hat v^n_k)_{k\in\Na}$ is a Markov chain.
\hfill\QED

Based on the two examples above, we observe some important properties of the empirical operator $\hat T^n_k$ defined in \eqref{eqn:hatTdiscount} and \eqref{eqn:hatTaverage}. First, consider the empirical Bellman operator for the discounted case in \eqref{eqn:hatTdiscount}. For every $v\in\ALP V$ and $\epsilon>0$, the random operators $(\hat T^n_k)_{k,n\in\Na}$ satisfy a \textit{consistency property} for every $k\in\Na$:
\beq{\label{eqn:pcp1}{\tt PCP}_1:~~ \lim_{n\rightarrow\infty}\pr{\|\hat T^n_k (v) - T(v)\|>\epsilon} = 0,}
which is established in in Proposition 5.1 in \cite{haskell2016}. In addition to this, it is not difficult to prove that every realization of $\hat T^n_k$ is a contraction operator over the space $\ALP V$ with contraction coefficient $\alpha$. 

Next, we consider the empirical Bellman operator $\hat T^n_k$ used in empirical relative value iteration in \eqref{eqn:hatTaverage}. This operator satisfies the property stated in \eqref{eqn:pcp1}, where the norm is replaced with the span seminorm. Let $\hat\alpha^n_k$ denote the contraction coefficient of $\hat T^n_k$. We can show that the random operators $(\hat T^n_k)_{k,n\in\Na}$ satisfy the following properties in the limit:
\beq{\label{eqn:pcp2}{\tt PCP}_2: &\lim_{n\rightarrow\infty} \pr{\hat\alpha^n_k >1-\delta} = 0 \text{ for all } \delta\in(0,1-\alpha), k\in\Na \nonumber\\
& \text{ and } \hat\alpha^n_k\leq 1 \text{ almost surely for all } n\in\Na, k\in\Na.}
We note here that for random operators, the Lyapunov exponent is defined as $\ex{\log(\hat\alpha^n_k)}$. Thus, the random operators enjoy \textit{negative Lyapunov exponent property} in the limit as $n\to\infty$.

We refer to the two conditions in \eqref{eqn:pcp1}-\eqref{eqn:pcp2} as the probabilistic contraction properties of the random operators. In the next section, we show that these two properties are crucial (in addition to some other hypotheses) in establishing the following result: For any $\kappa>0$, 
\beq{\label{eqn:hatxxstar}\lim_{n\to\infty}\limsup_{k\to\infty} \pr{\rho(\hat X^n_k,x^\star)>\kappa} = 0.}
This is true even when $\ALP X$ is a complete metric space. We then illustrate the application of the general theory developed here to MDPs with continuous state and action spaces. Fitted value iteration with function approximation for discounted cost MDPs with compact state and action spaces are considered in Subsection \ref{sub:discount}. The average cost case is treated in Subsection \ref{sub:average}. Theorems \ref{thm:evidiscounted} and \ref{thm:eviaverage} establish property \eqref{eqn:hatxxstar} for these cases.

\subsection{Related Work}
The convergence of random iterates generated from randomized algorithms has been a subject of intense investigation in the past 50 years. The random iterates can either converge to a deterministic quantity, admit an invariant distribution, or occupy certain open balls with high probability asymptotically. These can be construed as the stability property of the random iterates. There are three general methods to determine the stability of the random iterates: the Foster-Lyapunov method, the martingale method, or the ordinary differential equation method. The goal of this paper is to develop a fourth method based on the stochastic dominance approach. 

To invoke the Foster-Lyapunov method \cite{meyn2012,borovkov1998}, one requires the existence of a coercive Lyapunov function (i.e., whose sublevel sets are compact). If the random operators are independent and satisfy certain drift conditions with respect to the Lyapunov condition (called Foster-Lyapunov conditions), then it can be shown that under reasonable assumptions, the Markov chain admits an invariant distribution; we refer the reader to the texts \cite{meyn2012,douc2018markov} where this approach is described in detail. In finite dimensional problems, such Lyapunov functions are rather easy to construct. For infinite dimensional spaces, it is rather difficult to devise a coercive Lyapunov function. Thus, this approach is difficult to apply when the underlying state space of the random iterates is a function space.

The martingale and the ordinary differential equation approaches to establishing stability of the random iterates are quite general\cite{borkar2000ode,borkar2009stochastic,huang2002ode,kushner2003stochastic} -- in these approaches, we do not require the sequence of random operators to be independent. We only need certain additive errors to be a martingale noise difference sequence with respect to some filtration. This method is applicable mostly to finite dimensional spaces \cite{kushner2003stochastic}. We refer the reader to recent papers \cite{kulkarni2009finite, dieuleveut2016nonparametric} and the references therein for a discussion of this approach to Hilbert spaces. Once again, this approach is difficult to apply over general function spaces.

In contrast, the probabilistic contraction approach developed here can be used when the space of iterates is a complete metric space. However, to derive the convergence result, we assume that the random operators are independent and identically distributed, which is more restrictive that the assumptions placed in the martingale and the ordinary differential equation approaches. We further require this sequence of operators to satisfy the consistency and the negative Lyapunov exponent property (see \eqref{eqn:pcp1} and \eqref{eqn:pcp2} above). We call this approach the {\it probabilistic contraction analysis} to distinguish it from the previously studied approaches.


\subsection{Our Contribution}
The main result establishes the probabilistic fixed point of the random operators under certain reasonable assumptions. We prove this result using a stochastic dominance approach. The key benefits of the new approach we develop are:
\begin{enumerate}
 \item Our approach works when the underlying space $\ALP X$ is any complete metric space without a need to construct a Lyapunov function.
  \item Since we use stochastic dominance approach of the error process $\rho(\hat X^n_k,x^\star)$ to derive the probabilistic fixed point property, we avoid certain measurability issues that would typically arise if we were working with the random operators in the original state space $\ALP X$.
  \item If $\ALP X$ is a bounded complete metric space, then we obtain the rate of convergence and the sample complexity of getting close to the probabilistic fixed point with high probability. 
\end{enumerate}


The outline of the paper is as follows. In Section \ref{sec:main}, Theorem \ref{thm:hatt} bounds  $\limsup_{k\rightarrow\infty}\pr{\rho\Big(\hat X^n_k,x^\star\Big)\geq \kappa}$ for $\kappa>0$. This result in proved in Section \ref{sec:proofs}, where we use the theory of stochastic dominance \cite{shaked2007} to derive an upper bound on the probability of error being larger than $\kappa$ in the limit $k\rightarrow\infty$. We apply the results to study the convergence of the empirical value iteration algorithm for an MDP with the average cost criterion. In Subsection \ref{sub:composition}, we consider the case where the contraction operator comprises of iterated composition of many non-expansive mappings. We identify some sufficient conditions on the individual mappings and their randomized counterparts, so that the composite operator and their randomized versions satisfy the assumption of Theorem \ref{thm:hatt}. In Subsection \ref{sub:sample}, we present a sample complexity result to obtain a high confidence solution for probabilistic fixed point. We apply these results to determine the convergence of fitted value iteration for discounted and average cost MDPs with compact state and action spaces in Section \ref{sec:mdps}. We present proofs of the main results in Section \ref{sec:proofs}. Finally, we present some concluding thoughts in Section \ref{sec:conclusion}.

\section{Main Results}\label{sec:main}
Let us now formulate the problem precisely. Let $(\Omega,\FLD F,\mathbb{P})$ be a standard probability space and $\ALP X$ be a complete metric space. Define $\hat T^n_k:\ALP X\times\Omega\to\ALP X$ over this probability space such that for each $n\in\Na$, the operators $\hat T^n_k(x)$ and $\hat T^n_{k'}(x)$ are independent of each other and identically distributed for $k\neq k'$ and for all $x\in\ALP X$ (we suppress here the dependence of $\hat T^n_k$ on $\omega$ for easing the exposition). Let $\hat \alpha^n_k:\Omega\to[0,1]$ and $\hat \zeta^n_k:\Omega\to[0,\infty)$ be measurable functions such that 
\beq{ \label{eqn:alphank1}\rho(\hat T^n_k(x_1),\hat T^n_k(x_2))\leq \hat\alpha^n_k \rho(x_1,x_2) + \hat \zeta^n_k,}
for all $x_1,x_2\in\ALP X$ and $n,k\in\Na$. Here, $\hat \zeta^n_k$ could be bias and errors introduced due to certain approximations (such as function approximation or biased approximation of operators). Due to independence of the operators $(\hat T^n_k)_{k\in\Na}$, for every $n\in\Na$, $(\hat \alpha^n_k,\hat \zeta^n_k)_{k\in\Na}$ is an i.i.d. tuple. 

If $\hat\zeta^n_k\equiv 0$, then $\hat\alpha^n_k$ denotes a measurable upper bound on the contraction coefficient of $\hat T^n_k$. In this case, $\hat\alpha^n_k$ can be defined as
\beq{\hat \alpha^n_k = \sup_{x_1\neq x_2} \frac{\rho(\hat T^n_k(x_1),\hat T^n_k(x_2))}{\rho(x_1,x_2)}. \label{eqn:alphank2}}
However, it is not obvious that $\hat\alpha^n_k$ thus defined is measurable since $\ALP X$ is uncountable. It has been shown in \cite[Lemma 5.1]{diaconis1999iterated} that if $\ALP X$ is separable and complete, then $\hat \alpha^n_k$ as defined in \eqref{eqn:alphank2} is indeed measurable. For our applications, it is generally easy to derive a measurable map $\hat \alpha^n_k$ such that \eqref{eqn:alphank1} holds.

There is no reason to believe that for any sequence of random operators, a probabilistic fixed point would exist. However, we have already seen two conditions that are met in the learning algorithms that inspired this work: the consistency condition in \eqref{eqn:pcp1} and the negative Lyapunov exponent condition in \eqref{eqn:pcp2}. Our primary interest in this paper is in identifying further restrictions on the random operators under which probabilistic fixed point exists. We show that the probabilistic fixed point of $(\hat T^n_k)_{n,k}$ coincides with the fixed point of $T$. First, we need to place the following conditions on the random operators.
\begin{assumption}\label{assm:hatt}
The following holds:
\begin{enumerate}[(i)]
 \item $\ALP X$ is a complete metric space.
 \item $T:\ALP X\rightarrow\ALP X$ is a contraction operator with contraction coefficient $\alpha<1$. Let $x^\star$ denote the fixed point of $T$. 
 \item For each $n\in\Na$, $(\hat T^n_k)_{k\in\Na}$ is a sequence of independent and identically distributed operators.
 \item For any $k\in\Na$ and $\delta\in (0,1-\alpha)$,
 \beqq{\lim_{n\rightarrow\infty}\pr{\hat\alpha^n_{k}\geq 1-\delta} = 0}
 and $\hat\alpha^n_k\leq 1$ almost surely for all $n,k\in\Na$.
 \item For any $k\in\Na$ and $\epsilon>0$,
 \beqq{\lim_{n\rightarrow\infty}\pr{\rho\Big(\hat T^n_k(x^\star), T(x^\star)\Big) +\hat\zeta^n_k\geq \epsilon} = 0.}
 \item There exists $\bar w>0$ such that $\rho(\hat T^n_k(x^\star),T(x^\star))+\hat \zeta^n_k\leq \bar w$ almost surely for every $n\in\Na$ and $k\in\Na$.
\end{enumerate}
\end{assumption}
 
In our experience, Assumption \ref{assm:hatt} can be ascertained in practice as follows. Assumption \ref{assm:hatt}(i)-(iii) is routine and easy to ascertain in most cases. Assumption \ref{assm:hatt}(iv)-(v) is typically established using a concentration of measures result \cite{ledoux2001concentration,douc2018markov} or empirical processes theory \cite{pollard1984convergence,van1996weak}. Computing tight bounds on $\bar w$ used in Assumption \ref{assm:hatt}(vi) is the one that requires some effort. Based on the assumption, we have the following result, whose proof is presented in Section \ref{sec:proofs}.

\begin{theorem}\label{thm:hatt}
Suppose that Assumption \ref{assm:hatt} holds. Fix $\kappa>0$ and pick $\epsilon,\delta>0$ in the set
\beq{\ALP B:=\left\{\epsilon\in\left(0,\frac{\kappa(1-\alpha)}{2}\right],\delta\in(0,1-\alpha]: \left\lceil\frac{2}{\delta}\right\rceil \leq \frac{\kappa}{\epsilon}\right\}.\label{eqn:ALPB}}
Let $\beta = \left\lceil\frac{\bar w}{\epsilon}\right\rceil$. Pick $n$ sufficiently large so that $p^n_{\epsilon,\delta}>\frac{2\beta}{2\beta+1}$ can be picked satisfying
\beq{p^n_{\epsilon,\delta}\leq \inf_{k\in\Na} \pr{\hat\alpha^n_k\leq 1-\delta,\rho\Big(\hat T^n_k(x^\star) , T(x^\star)\Big)+\hat\zeta^n_k \leq \epsilon}.\label{eqn:pned}}
Then, we have 
\beqq{\limsup_{k\rightarrow\infty}\;\pr{\rho\Big(\hat X^n_k,x^\star\Big)>\kappa} &\leq \frac{1-(p^n_{\epsilon,\delta})^\beta}{\big(p^n_{\epsilon,\delta}\big)^\beta}.}
This further implies
\beqq{\lim_{n\rightarrow\infty}\limsup_{k\rightarrow\infty}\;\pr{\rho\Big(\hat X^n_k,x^\star\Big)>\kappa} &=0.}
Consequently, the point $x^\star$ is the unique probabilistic fixed point of the random operators $(\hat T^n_k)_{k\in\Na,n\in\Na}$.
\end{theorem}

Let us illustrate the application of the theorem by applying it  on Example 1.

\noindent\textit{Example 3:}
Recall the notation in Example 1. Here, we have $\ALP X = \Re^{\ALP S}$ endowed with the $\ell_\infty$ norm, $\hat\alpha^n_k \equiv\alpha$, and $\zeta^n_k = 0$ almost surely. We have
\beqq{&\|\hat T^n_k (v^*) - T(v^*)\|_\infty = W^n_k\leq \\
&\alpha \max_{s,a}\Bigg|\frac{1}{n}\sum_{i=1}^nv^*(g(s,a,Z_{k,i})) - \ex{v^*(g(s,a,Z))}\Bigg|.}
Using the Hoeffding inequality and the union bound, we have 
\beqq{\pr{W^n_k>\epsilon}\leq 2|\ALP S||\ALP A| \exp\left(-\frac{n\epsilon^2}{2\alpha^2\|v^*\|_\infty^2}\right).}
Further, $W^n_k \leq \bar w :=2\frac{\|c\|_\infty}{1-\alpha}$. Pick $\delta = 1-\alpha$, $\epsilon = \frac{\kappa(1-\alpha)}{2}$, we conclude that $v^*$ is a probabilistic fixed point of the operators $(\hat T^n_k)_{n,k}$. This result was derived in \cite[Theorem 3.1]{haskell2016} \new{by exploiting the structure of the problem. Here, we have used the general result of Theorem \ref{thm:hatt} to arrive at the conclusion.}

We would like to note two interesting observations here. The first one is that under this assumption, we require the consistency condition, Assumption \ref{assm:hatt}(v), to hold only at $x^\star$, not over the entire space. We can exploit some specific property of $x^\star$ to establish the consistency result. For instance, in a class of MDPs, the space of value functions $\ALP V$ could be the space of bounded measurable functions, but the optimal value function $v^\star$ could be a Lipschitz continuous function (see, for example, \cite{hinderer2005lipschitz,shao2022ev}). Thus, we do not need to show the consistency condition for all bounded measurable functions, but just for one Lipschitz continuous function $v^\star$.

The second observation is that $\bar w$, defined in Assumption \ref{assm:hatt}(vi), can be made small by an appropriate choice of the random operators $(\hat T^n_k)_{k\in\Na}$. For instance, almost all the variance reduction algorithms we have seen in the literature, for example \cite{defazio2014saga,harikandeh2015stopwasting,sidford2018near,wainwright2019variance}, attempt to make $\bar w$ as small as possible by a carefully constructed random operator $\hat T^n_k$. To see this, note that if $\bar w = 0$ and $n$ is sufficiently large so that the random operators have a negative Lyapunov exponent, then 
 $\beta = 0$ and for any $\kappa>0$, we have
\beqq{\limsup_{k\rightarrow\infty}\;\pr{\rho\Big(\hat X^n_k,x^\star\Big)>\kappa} &\leq \frac{1-(p^n_{\epsilon,\delta})^\beta}{\big(p^n_{\epsilon,\delta}\big)^\beta} = 0.}
We present a more complete discussion on this part with several examples in a recent paper of the first author \cite{gupta2021convergence}.

\begin{remark}\label{rem:lbalphaw}
A lower bound in \eqref{eqn:pned} can be derived from Frechet-Hoeffding Copula bound \cite[Theorem 3.1.1]{rachev1998i}:
\begin{align*}
& \mathbb{P}(\hat{\alpha}_k^n\leq
1-\delta,\rho(\hat{T}_k^n(x^*),T(x^*))+\hat{\zeta}_k^n\le \epsilon)
\\
& \geq 
\mathbb{P}(\hat{\alpha}_k^n
\leq1-\delta) + \mathbb{P}(\rho(\hat{T}_k^n(x^*),T(x^*))+\hat{\zeta}_k^n\le \epsilon) - 1\\
& = 1-\mathbb{P}(\hat{\alpha}_k^n>
1-\delta)-\mathbb{P}(\rho(\hat{T}_k^n(x^*),T(x^*))+\hat{\zeta}_k^n> \epsilon) 
\end{align*}
The first term on the right can be derived from the usual concentration of measure results. The second term on the right can be derived from covering number argument frequently used in empirical process theory \cite{van1996weak,gyorfi2002principles} or function approximation theorems \cite{hanin2019universal,steinwart2008support}. These ideas have been employed in Section \ref{sec:mdps} of the paper as well. 
\end{remark}

We now proceed to develop two other consequences of the main result. In the next subsection, we consider the case when the contraction operator is a composition of multiple maps, which are approximated by random maps to yield the corresponding random operators. In the subsequent subsection, we derive the sample complexity result for the convergence to probabilistic fixed point.  
\subsection{Composition of Operators}\label{sub:composition}
In many complex optimization and reinforcement learning problems, the operator $T$ may be a composition of multiple mappings, each of which may be non-expansive or contraction. We may need to replace each of these mappings using random approximations. Our goal in this section is to study such cases and identify sufficient conditions on the random mappings that implies Assumption \ref{assm:hatt}. 

To be precise, let $\ALP X_1,\ldots,\ALP X_{L+1}$ be $L+1$ metric spaces with $\ALP X_1 = \ALP X_{L+1} = \ALP X$, which is, by assumption, a complete metric space. Let $\rho_l$ denote the metric on the space $\ALP X_l$. Let $T_l:\ALP X_l\rightarrow\ALP X_{l+1}$, $l=1,\ldots,L$ be a sequence of mappings such that $T_l$ is Lipschitz with coefficient $\alpha_l\in[0,1]$, i.e.,
\beqq{\rho_{l+1}(T_l(x),T_l(y))\leq \alpha_l\rho_l(x,y).}
Define $T:=T_L\circ \ldots\circ T_1$, and let $\alpha:=\prod_{l=1}^L\alpha_l$ denote the contraction coefficient of $T$ (we assume here that $\alpha<1$). We do not assume completeness of $\ALP X_2,\ldots,\ALP X_L$. Let $x^\star$ be the fixed point of $T$. Define $x_1^\star = x^\star$ and recursively define
\beqq{x_{l+1}^\star = T_l(x_l^\star),\quad  l = 1,\ldots,L.}
It is obvious that $x_{L+1}^\star = x^\star$.

Let $\hat H^n_{k,l}:\ALP X_l\rightarrow\ALP X_{l+1}$ be the random mapping at time $k$ that replaces the map $T_l$ (here, $l=1,\ldots,L$). Define $\hat T^n_k:=\hat H^n_{k,L}\circ\ldots\circ \hat H^n_{k,1}$ to the random operator that approximates $T$. Let $\hat \alpha^n_{k,l}:\Omega\to[0,1]$ denote the measurable map such that $\hat H^n_{k,l}$ satisfies for $l = 1,\ldots, L-1$
\beqq{\rho_{l+1}\Big(\hat H^n_{k,l}(x_{l,1}),\hat H^n_{k,l}(x_{l,2})\Big) & \leq \hat\alpha^n_{k,l}\rho(x_{l,1},x_{l,2}),\\
\rho_{L+1}\Big(\hat H^n_{k,L}(x_{L,1}),\hat H^n_{k,L}(x_{L,2})\Big)& \leq \hat\alpha^n_{k,L}\rho(x_{L,1},x_{L,2})+\hat\zeta^n_k,}
for all $x_{l,1},x_{l,2}\in\ALP X_l, l = 1,\ldots,L$. We make the following assumption on the random mappings.
\begin{assumption}\label{assm:hath}
 The following holds:
 \begin{enumerate}[(i)]
  \item For each $l=1,\ldots,L$, the sequence $(\hat H^n_{k,l})_{k\in\Na}$ is a sequence of independent and identically distributed mappings. Each mapping $\hat H^n_{k,l}$ is independent of all of the other mappings.
 \item For every $n,k\in\Na$, $\hat\alpha^n_{k,l} \leq 1$ almost surely. Further, for every $k\in\Na$ and $\delta\in (0,1-\alpha)$,
 \beqq{\lim_{n\rightarrow\infty}\pr{\min_{l}\hat\alpha^n_{k,l}\geq 1-\delta} = 0.}
 \item For every $l = 1,\ldots,L-1$, every $k\in\Na$, and $\epsilon>0$,
 \beqq{\lim_{n\rightarrow\infty}\pr{\rho\Big(\hat H^n_{k,l}(x^\star_l), T_l(x^\star_l)\Big)\geq \epsilon} = 0,\\
 \lim_{n\rightarrow\infty}\pr{\rho\Big(\hat H^n_{k,L}(x^\star_L), T_L(x^\star_L)\Big)+\hat \zeta^n_k\geq \epsilon} = 0.}
 \item For each $l=1,\ldots,L$, there exists $\bar w_l>0$ such that $\rho\Big(\hat H^n_{k,l}(x^\star_l),T_l(x_l^\star)\Big)\leq \bar w_l$ for $l=1,\ldots,L-1$ and $\rho\Big(\hat H^n_{k,L}(x^\star_L),T_L(x_L^\star)\Big)+\hat\zeta^n_k\leq \bar w_L$ almost surely for every $n\in\Na$ and $k\in\Na$.
 \end{enumerate}
\end{assumption}

\begin{theorem}\label{thm:hath}
 If Assumption \ref{assm:hath} is satisfied, then the composite operators $T$ and $\hat T^n_k$ satisfies Assumption \ref{assm:hatt}.
\end{theorem}
\begin{proof}
It is clear that by hypotheses, Assumption \ref{assm:hatt}(i)-(ii) is satisfied. We refer the reader to Subsection \ref{sub:hath}, where we establish that Assumption \ref{assm:hatt}(iii)-(vi) is also satisfied.
\end{proof}

Typically, all $(\hat H^n_{k,l})_{l=1}^L$ may not use the same number of samples $n$. In these cases, we can replace the number of samples $n_l$ as a function of $n$ such that $n_l(n)$ monotonically diverges to infinity as $n$ diverges to infinity.

\subsection{Sample Complexity of Computing Probabilistic Fixed Point}\label{sub:sample}
To establish Theorem \ref{thm:hatt}, we  employ the theory of stochastic dominance. We construct a dominating Markov chain over the space of natural numbers whose invariant distribution dominates the probability of error being large as $k\to\infty$. The dominating Markov chain construction can provide us with a sample complexity bound on the number of samples required to get close to the fixed point $x^\star$ with high probability. For a fixed $\kappa>0$ and confidence level $\xi$, let $N_{\kappa,\xi}$ be defined as
\beqq{N_{\kappa,\xi} =&  \inf_{(\epsilon,\delta)\in\ALP B} \Bigg\{n\in\Na:\\
& \pr{\hat\alpha^n_k\leq 1-\delta,W^n_k \leq \epsilon} \geq \sqrt[w]{\frac{1}{1+\xi}} \Bigg\},}
where $\ALP B$ is defined in \eqref{eqn:ALPB}, and we define $W^n_k:=\rho\Big(\hat T^n_k(x^\star) , T(x^\star)\Big)+\hat\zeta^n_k$. Then, for any $n\geq N_{\kappa,\xi}$, we have
\beqq{\underset{k\rightarrow\infty}{\lim\sup}\;\pr{\rho(\hat X^n_k , x^\star)\geq \kappa}\leq \xi.}
The proof of the above result follows immediately from Theorem \ref{thm:hatt} and Proposition \ref{PROP:INVARIANT}, which is established later in Section \ref{sec:proofs}. We note here that the sample complexity is only in terms of the number of samples $n$ picked to approximate the operator at every time step; we allow the number of iterations to go to infinity, which implies that the true sample complexity is infinite. 


To provide a finite sample complexity bound, we further need to bound the space as well, that is, we assume $\ALP X$ to be a bounded space satisfying:
\beqq{D:=\text{diam}(\ALP X) = \sup_{x,y\in\ALP X}\rho(x,y)<\infty.}
This condition is readily satisfied in many optimization and empirical dynamic programming algorithms, where the operator $\hat T^n_k$ is a contraction operator with the contraction coefficient $\alpha$ almost surely, and one knows the region in the space $\ALP X$ where the solution lies. 

This is also satisfied if we allow the random operator to project the iterates back to $\ALP X$ if the iterates become large. In this case, we have the composition of a projection operator and a random operator. It is well-known that if $\ALP X$ is a Hilbert space (say Euclidean space with the $\ell_2$ norm or a reproducing kernel Hilbert space), then the projection operation is non-expansive. In case $\ALP X$ is not a Hilbert space, then the projection operation must be picked carefully so that it remains a non-expansive operator. In this situation, the property of the random operators remain unchanged due to the projection operation.

In the following theorem, we present a sample complexity bound for this case, which is based on the analysis in \cite{haskell2016} and has been used in \cite{haskell2017randomized,haskell2019universal,sharma2019approximate,sharma2019approximately, sharma2019empirical} to provide the sample complexity bounds for various empirical dynamic programming algorithms with function approximation. For $\kappa>0$, pick $(\epsilon,\delta)\in\ALP B$, where $\ALP B$ is defined in \eqref{eqn:ALPB}. Let $d := \lceil D/\epsilon\rceil$, $p:=p^n_{\epsilon,\delta}$, which satisfies \eqref{eqn:pned}, and $\eta:=\left\lceil\frac{2}{\delta}\right\rceil $. For confidence level $\xi>0$, define $N_{\kappa,\xi}$ and $K_{\kappa,\xi}$ as 
\beqq{N_{\kappa,\xi} &:= \inf \Bigg\{n\in\Na:p \geq\max\left\{\frac{1}{2},\left( 1-\frac{\xi}{3} \right)^{1/(d-\eta-1)}\right\} \Bigg\},}
\beqq{K_{\kappa,\xi} &:= \log\left( \frac{3}{\xi(1-p)p^{d-\eta-1}}\right).}

\begin{theorem}\label{thm:samplecomplex}
Suppose that Assumption \ref{assm:hatt} (i)-(v) holds with a bounded complete metric space $\ALP X$ with diameter $D$. Then, we have
 \beqq{\pr{\rho(\hat X^n_k,x^\star)>\kappa}< \xi\text{ for all } k\geq K_{\kappa,\xi} \text{ and }  n\geq N_{\kappa,\xi}.}
\end{theorem}
\begin{proof}
See Subsection \ref{sub:samplecomplex}.
\end{proof}

\new{We note here that for computing $N_{\kappa,\xi}$, one can use the lower bound stated in Remark \ref{rem:lbalphaw} to derive the sample complexity bound. The term $\pr{\hat{\alpha}_k^n> 1-\delta}$ can be bounded either using Hoeffding inequality or may not need to be bounded at all if $\hat{\alpha}_k^n$ is a constant. On the other hand, $\pr{\rho(\hat{T}_k^n(x^*),T(x^*))+\hat{\zeta}_k^n> \epsilon}$ can be bounded by bounding the biasedness introduced due to the random operator and using the empirical process theory, which requires knowledge of the covering number of the space $\{\hat T^n_k(x^*)\}$. Sample complexity results for various random operators defined the context of empirical value iteration are discussed in \cite{sharma2019approximate,haskell2016,haskell2017randomized,haskell2019universal,gupta2015edp,munos2008finite}, among several others.}

\section{Application to Fitted Value Iteration}\label{sec:mdps}
We now apply the general framework developed above to derive the probabilistic fixed point property of the fitted value iteration for continuous state continuous action MDPs. In this setting, the MDP model and the cost function are known; further, there is a generative model (a simulator) of the MDP that generates actuation noises which can be used to generate the next state and the cost incurred given the current state-action pair. Fitted value iteration is very useful for approximately solving MDPs where the state and the action spaces are high dimensional (50 or more dimensional). This problem has been studied under different names within the reinforcement learning literature. References 
\cite{kakade2003exploration,lattimore2012pac,szorenyi2014optimistic,sidford2018near} consider a finite-state finite-action setting, and refer to this algorithm as an MDP with a generative model (no function approximation is used here). Reference \cite{munos2008finite} consider the MDPs with a continuous-state finite-action setting with function approximation and refers to this algorithm as fitted value iteration. In our previous work, this algorithm has been referred to as empirical value iteration \cite{gupta2015edp,haskell2016,haskell2017randomized,haskell2019universal,sharma2019approximate,sharma2019approximately, sharma2019empirical}. In these previous works, a variety of sufficient conditions are determined under which the optimal value function of the MDP is the probabilistic fixed point of the random operators. In this section, we use the techniques developed here to show essentially the same results, albeit under different assumptions on the MDPs. 

\subsection{Fitted Value Iteration}
We consider here the same model as introduced in Examples 1 and 2 in Section \ref{sec:intro}, but with a compact state and action setting. In this setting, the value functions belong to a function space, and function approximators such as neural networks or non-parametric regression are needed to store the value function in a computer. Thus, we need to project the value functions to appropriate function approximating class. In the sequel, we used both the transition probability matrix $P(ds'|s,a)$ induced from the state space model $s' = g(s,a,z)$. 

Recall the Bellman operator $T$ from \eqref{eqn:Tdiscount} in Examples 1. In the context of fitted value iteration, we view the operator $T$ as a composition of an identity operator and the Bellman operator. Since a continuous function is difficult to store in a computer, one usually uses a parametric or a non-parametric function approximator to store the value function; such a projection approximates the identity operator. On the other hand, in empirical value iteration, the Bellman operator is approximated using the empirical mean, as is done in \eqref{eqn:hatTdiscount}.

Let $\ALP L_\infty(\ALP S)$ denote the set of all bounded measurable functions over $\ALP S$ endowed with the supremum norm and $\ALP V\subset\ALP L_\infty(\ALP S)$ be the set of bounded continuous functions. We define the map $\hat H^n_k:\ALP V\rightarrow\ALP L_\infty(\ALP S)$ as  
\beq{\label{eqn:hathkn}\hat H^n_k(v)(s):= \min_{a\in\ALP A} \left(c(s,a)+ \frac{\alpha}{n} \sum_{i=1}^{n} v(g(s,a,Z_{k,i}))\right).}
We note here that the range of $\hat H^n_k$ is the space of measurable functions. However, depending on the continuity property of $g$, it may well be the case that the range of $\hat H^n_k$ is within the class of continuous and bounded function $\ALP V$.

For the average cost case, we define the empirical relative Bellman operator $\hat G^n_k:\ALP V\rightarrow\ALP L_\infty(\ALP S)$ as  
\beqq{\tilde v(s) & = \min_{a\in\ALP A} \left(c(s,a)+ \frac{1}{n} \sum_{i=1}^{n} v(g(s,a,Z_{k,i}))\right)}
\beqq{
\hat G^n_k(v)(s) & = \tilde v(s) - \inf_{s'\in\ALP S}\tilde v(s').}
The operator $\hat G^n_k$ is the empirical relative value iteration operator, which was introduced in our earlier work \cite{gupta2015edp}.
 
Let $\hat \Pi^m_k:\ALP L_\infty(\ALP S)\rightarrow \ALP V$ denote the projection operator that maps the output of $\hat H^n_k$ to a continuous and bounded function using a parametric regression model \cite{haskell2017randomized,haskell2019universal} or a non-parametric regression model \cite{shah2018q,sharma2019approximate}. Typical parametric regression functions include neural networks \cite{gyorfi2002principles,hanin2019universal} or reproducing kernel Hilbert spaces (RKHS) \cite{steinwart2008support}. For non-parametric regression functions, nearest neighbor or kernel estimates are used, as done in \cite{shah2018q,sharma2019approximate}.

In practice, the projected function $\hat \Pi^m_k(v)$ is constructed as follows: Let $\ALP F_m\subset\ALP V$ be a function approximating class. Here, $m$ can represent the number of data points used for nearest neighbor function approximation, number of kernels or data points for kernel-based function approximators, the number of basis functions used in RKHS, or the number of perceptrons in the neural network. The algorithm generates $m$ uniformly distributed i.i.d. samples of the states $\{s_{k,i}\}_{i=1}^m$ and creates a dataset $\{(s_{k,i},v(s_{k,i})\}_{i=1}^m$. This dataset is fitted as
\beqq{\hat\Pi^m_k(v) = \underset{f\in\ALP F_m}{\arg\min}\; L(v,f),}
where $L:\ALP V\times\ALP F_m\to\Re$ is a loss function that measures some notion of the distance between $v\in\ALP V$ and $f\in\ALP F_m$. A typical loss function is the usual weighted $\ell_2$-norm with weights $(\tau_i)_{i=1}^m$:
\beqq{L(v,f):=\frac{1}{m} \sum_{i=1}^m \tau_i\Big(v(s_{k,i}) - f(s_{k,i})\Big)^2.}
Due to the well-known universal function approximating properties of the aforementioned function classes under supremum norm, one can show that $\hat \Pi^m_k$ satisfies certain reasonable assumptions that are listed in Assumption \ref{assm:project} below.

Let $n_1(n)$ and $n_2(n)$ be monotonically increasing functions of $n$ taking values in the space of natural numbers. The fitted value iteration operator is defined as $\hat T^n_k(v) := \hat \Pi^{n_2(n)}_k\circ \hat H^{n_1(n)}_k(v)$. The relative fitted value iteration map is defined as $\hat T^n_k(v) := \hat \Pi^{n_2(n)}_k\circ \hat G^{n_1(n)}_k(v)$. We study their probabilistic fixed point properties next.

\subsection{Fitted Value Iteration for Discounted MDPs}\label{sub:discount}
Let $\rho_S$ denote a metric on $\ALP S$, $\rho_A$ denote a metric on $\ALP A$ and $\rho_{S\times A}$ denote the metric on $\ALP S\times\ALP A$ defined as 
  \beqq{\rho_{S\times A}((s,a),(s',a')) = \rho_S (s,s')+\rho_A(a,a').}
We place the following assumption on the MDP, which is motivated from \cite{hinderer2005lipschitz}.
\begin{assumption}\label{assm:lipmdp}
The following holds:
 \begin{enumerate}[(i)]
  \item $\ALP S,\ALP A$ are compact subsets of Euclidean spaces.
  \item $g$ is Lipschitz continuous in $(s,a)\in\ALP S\times\ALP A$ for every $z\in\ALP Z$ with Lipschitz coefficient $L_g(z)$ and $L_P:=\int L_g(z)\pr{dz}<1/\alpha$;
  \item The cost function $c(s,a)$ is Lipschitz continuous with Lipschitz coefficient $L_c$.
 \end{enumerate}
\end{assumption}

We now make the following assumption on the projection operator.
\begin{assumption}\label{assm:project}
The projection operator $\hat \Pi^n_k:\ALP L_\infty(\ALP S)\rightarrow \ALP V$ satisfies two conditions:
\begin{enumerate}[(i)]
 \item $\hat \Pi^{n}_k$ is non-expansive with a bias, that is, $\|\hat \Pi^{n}_k(v_1) - \hat \Pi^{n}_k(v_2)\|_\infty\leq \|v_1-v_2\|_\infty+\hat\zeta^n_k$, where $\pr{\lim_{n\to\infty}\hat\zeta^n_k = 0 } = 1$ for all $k\in\Na$ and $\zeta^n_k\leq \bar\zeta$ for all $n,k\in\Na$ almost surely.
\item For any $\epsilon>0$, and $\delta>0$, there exists $N_2:=N_2(v^*,\epsilon,\delta)$ such that 
\beqq{\pr{\|\hat \Pi^{n}_k(v^*) - v^*\|_\infty>\epsilon}<\delta \text{ for all } n\geq N_2.}
\end{enumerate}
\end{assumption}
Consider Assumption \ref{assm:project}(i) and note the following inequality:
\begin{align*}
    &\|\hat \Pi^n_k(v_1) - \hat \Pi^n_k(v_2)\|_\infty \leq \|v_1 - v_2\|_\infty \\
    &+ \underbrace{\|\hat \Pi^n_k(v_1) - v_1\|_\infty + \|\hat \Pi^n_k(v_2) - v_2\|_\infty}_{\hat \zeta^n_k}
\end{align*}
For many function approximating class such as a neural network, k-nearest neighbor, and RKHS, the non-expansiveness of the projection operation holds under the limit where (a) number of samples used for function fitting goes to infinity, (b) the size of the function approximating class goes to infinity, (c) the function approximating class contains the value functions throughout the simulation. 

To establish that Assumption \ref{assm:project} (ii) holds in a problem, we need to exploit some property of the optimal value function (such as Lipschitz or smoothness \cite{hanin2019universal}). \new{In this case, if the loss function $L(\cdot,\cdot)$ is convex in the parameters used in the function approximating class $\ALP F_m$, then stochastic gradient descent could yield the minimizing function in the function approximating class (sometimes using decreasing stepsizes). For situations where the loss function is non-convex, the estimation of $\|\hat \Pi^n_k(v) - v\|_\infty$ may be difficult and remains an open problem.}

With these assumptions in place, we have the following result.

\begin{theorem}\label{thm:evidiscounted}
If Assumptions \ref{assm:lipmdp} and \ref{assm:project} hold, then $v^*$ is a probabilistic fixed point of $(\hat T^n_k)_{n,k\in\Na}$.
\end{theorem}
\begin{proof}
We show in Subsection \ref{sub:evidiscounted} that $\hat T^n_k$ satisfies Assumption \ref{assm:hatt}. The result then immediately follows from Theorem \ref{thm:hath} and Theorem \ref{thm:hatt}.
\end{proof}

The above result differs from \cite[Theorem 2, Corollary 4]{munos2008finite} and is applicable to MDPs with continuous action spaces. In \cite{munos2008finite}, the authors assume discounted-average concentrability of future-state distributions (see Assumption 2 in \cite{munos2008finite}) and conducted the analysis under $\ell_p$ norm. This assumption does not hold for any stable deterministic control system (since the transition kernel is not absolutely continuous with respect to any base measure). In our result above, we are not making this strong assumption on the transition kernel. However, we make strong assumptions on the projection operator being biased non-expansive operator with respect to the $\ell_\infty$ norm, with bias term converging to 0 as the sample size diverges and function approximating class expands. We further assume $\ell_\infty$ norm over the value function space. Consequently, Theorem \ref{thm:evidiscounted} can be applied to deterministic dynamical system as well (in which case the noise is deterministic). Thus, Theorem \ref{thm:evidiscounted} and \cite[Theorem 2, Corollary 4]{munos2008finite} complement each-other.

\subsection{Fitted Value Iteration for Average Cost MDPs}\label{sub:average}
We now turn our attention to average cost MDPs and follow the same steps as in the previous section. The relative value iteration algorithm, which is used to compute the relative value function in the average cost MDP, is presented in \cite{hernandez2012adaptive,hernandez2012discrete,hernandez2012further,almudevar2014approximate}. We place the following assumptions on the average cost MDP, which are inspired by \cite{feinberg2019reduction,feinberg2018reduction}.

\begin{assumption}\label{assm:isolated}
The following holds:
 \begin{enumerate}[(i)]
  \item $\ALP S,\ALP A$ are compact subsets of Euclidean spaces. Further, there exists an isolated point $s_0\in\ALP S$ satisfying $\inf_{s\in\ALP S\setminus \{s_0\}} \|s_0-s\|>0$.
  \item $g$ is Lipschitz continuous in $(s,a)\in\ALP S\times\ALP A$ for every $z\in\ALP Z$ with Lipschitz coefficient $L_g(z)$ and $L_P:=\int L_g(z)\pr{dz}<1$;  
  \item The cost function $c(s,a)$ is Lipschitz continuous with Lipschitz coefficient $L_c$.
  \item There exists $\alpha\in(0,1)$ such that
    \beqq{& \inf_{(s,a)\in\ALP S\times\ALP A} P(s_0|s,a)   = 1-\alpha \\
    \Longleftrightarrow & \sup_{(s,a)\in\ALP S\times\ALP A} P(\ALP S\setminus\{s_0\}|s,a)  =\alpha.}
\end{enumerate}
\end{assumption}

We focus our attention on empirical relative value iteration for the average cost MDP with an isolated state. 

\begin{theorem}\label{thm:eviaverage}
Suppose that the average cost MDP satisfies Assumption \ref{assm:isolated} with $L_P<1$ and Assumption \ref{assm:project} with $\|\cdot\|_\infty$ replaced with $\sp(\cdot)$, then  $v^*$ is a probabilistic fixed point of $(\hat T^n_k)_{n,k\in\Na}$.
\end{theorem}
\begin{proof}
We show in Subsection \ref{sub:eviaverage} that $\hat T^n_k$ satisfies Assumption \ref{assm:hatt}. The result then immediately follows from Theorem \ref{thm:hath} and Theorem \ref{thm:hatt}.
\end{proof}

To the best of our knowledge, the probabilistic fixed point property of fitted relative value iteration for average cost MDPs with continuous state and action spaces has not been established before. For finite state and action space, this property was established in our earlier work \cite{gupta2015edp}.

\subsection{Proof of Theorem \ref{thm:evidiscounted}}\label{sub:evidiscounted}
It is also clear that $\ALP V$, the space of continuous and bounded value functions over $\ALP S$ endowed with the sup norm, is a complete metric space and that the Bellman operator $T$ is a contraction operator with fixed point $v^*$. Due to the construction of the empirical Bellman operator, we conclude that $\hat T^n_k$ is independent and identically distributed. Every realization of the empirical Bellman operator $\hat H^n_k$ is a contraction operator with contraction coefficient $\alpha$ for all $n,k\in\Na$. This yields together with Assumption \ref{assm:project}(i) that $\hat T^n_k$ is almost surely a contraction for all $n,k\in\Na$. This further yields
\beqq{\|\hat T^n_k(v^*) - T(v^*)\|_\infty \leq 2\alpha \|v^*\|_\infty +\bar\zeta.}

Thus, to establish Theorem \ref{thm:evidiscounted}, we only need to show that Assumption \ref{assm:hatt}(v) is satisfied. We establish this fact next.
\begin{proposition}\label{prop:empiricalmdp}
If Assumption \ref{assm:lipmdp} is satisfied, then for any $\epsilon,\delta>0$, there exists $N_1:=N_1(v^*,\epsilon,\delta)$ such that
\beqq{\pr{\|\hat H^n_k(v^*) - T(v^*)\|_\infty\geq \epsilon}<\delta \text{ for all } n\geq N_1.}
\end{proposition}
\begin{proof}
If Assumption \ref{assm:lipmdp} holds, then $v^*$ is a Lipschitz continuous function with Lipschitz coefficient $L_{v^*} \leq \frac{L_c}{1-\alpha L_P}$ from \cite[Theorem 4.2(c), p.13]{hinderer2005lipschitz}. To establish the result, we need to introduce some notations. Let $\ALP D$ be the set of functions given by
\beqq{\ALP D & = \Big\{d: \ALP Z\rightarrow\Re: \exists (s,a) \text{ such that } d(z) = v^*(g(s,a,z))\Big\}\\
&  = \bigcup_{(s,a)} \Big\{v^*(g(s,a,\cdot))\Big\}.} 
We note that 
\beqq{\|& \hat H^n_k(v^*) - T(v^*)\|_\infty \\
=&  \max_{s\in\ALP S}\Bigg| \min_{a\in\ALP A} \left(c(s,a)+ \frac{\alpha}{n} \sum_{i=1}^{n} v^*(g(s,a,Z_{k,i}))\right) \\
& \qquad - \min_{a\in\ALP A} \left(c(s,a)+ \alpha \int v^*(g(s,a,z)) \pr{dz}\right)\Bigg|,\\
\leq & \sup_{s\in\ALP S,a\in\ALP A} \Bigg| \frac{1}{n} \sum_{i=1}^{n} v^*(g(s,a,Z_{k,i})) - \int v^*(g(s,a,z)) \pr{dz}\Bigg|,\\
 = & \sup_{d\in\ALP D} \Bigg| \frac{1}{n} \sum_{i=1}^{n} d(Z_{k,i}) - \int d(z) \pr{dz}\Bigg|.}
This immediately yields
\beqq{& \pr{\|\hat H^n_k(v^*) - T(v^*)\|_\infty>\epsilon} \\
& \leq \pr{\sup_{d\in\ALP D} \Bigg| \frac{1}{n} \sum_{i=1}^{n} d(Z_{k,i}) - \int d(z) \pr{dz}\Bigg|>\epsilon}.}
To obtain bounds on the right side of the equation above, we need to show that the bracketing number\footnote{We refer the reader to \cite{van1996weak} for a definition and discussion about bracketing number.} of the class of functions $\ALP D$ is finite for every $\epsilon>0$. The finiteness of bracketing number is established in \cite[Theorem 2.7.11, p. 164]{van1996weak}. To see this, let us write $d_{s,a}(z) = v^*(g(s,a,z))$. Since $v^*$ and $g$ are Lipschitz continuous functions, we have 
\beqq{|d_{s,a}(z) - d_{s',a'}(z)|\leq L_{v^*}L_g(z)\;\rho_{S\times A}((s,a),(s',a')).}
This immediately implies that the bracketing number of $\ALP D$ is upper bounded by the covering number of $\ALP S\times\ALP A$ by \cite[Theorem 2.7.11, p. 164]{van1996weak}. Since $\ALP S\times\ALP A$ is compact, it is totally bounded; thus, its covering number is finite for any $\epsilon>0$. 


We now invoke \cite[Theorem 2.4.1, p. 122]{van1996weak} to conclude that for any $k\in\Na$,
\beqq{\lim_{n\rightarrow\infty}\pr{\sup_{d\in\ALP D} \Bigg| \frac{1}{n} \sum_{i=1}^{n} d(Z_{k,i}) - \int d(z) \pr{dz}\Bigg|>\epsilon} = 0,}
which yields the desired result.
\end{proof}

\begin{remark}
If one knows the exact covering number of $\ALP S\times\ALP A$ under the metric $\rho_{S\times A}$, then the sample complexity bound can also be given. For more information on covering and bracketing numbers, we refer the reader to \cite{van1996weak,pollard1984convergence,anthony2009neural}. Once we know the bracketing number of the function class $\ALP D$ under the sup norm, then we can apply Lemma 9.1 of \cite{gyorfi2002principles} to derive the sample complexity result for the EDP algorithm.
\end{remark}

\subsection{Proof of Theorem \ref{thm:eviaverage}}\label{sub:eviaverage}

The proof of this result follows from minor modifications in the proof of Theorem \ref{thm:evidiscounted} above. We define $\ALP V/\sim$ to be the quotient space 
\beqq{\ALP V/\sim = \Big\{[v]\subset C_b(\ALP S)& : v_1,v_2\in[v] \text{ implies } \\
& v_1-v_2 \text{ is a constant function} \Big\}.}
Under these assumptions, we establish the following result.
\begin{proposition}\label{prop:avgmdp}
The quotient space $\ALP V/\sim$ is a Banach space. If Assumption \ref{assm:isolated} holds, then we have
\begin{enumerate}[(a)]
 \item $T:\ALP V/\sim\to\ALP V/\sim$ is a contraction operator with coefficient $\alpha$ and $[v^*]\in\ALP V/\sim$. 
 \item Any representative element $\tilde v^*\in[v^*]$ satisfies $\sp(\tilde v^*)\leq \frac{\sp(c)}{1-\alpha}$.
 \item If in addition, Assumption \ref{assm:lipmdp} holds with $L_P<1$, then any representative element $v^*$ that lies within the unique equivalence class of the fixed points of $T$ is Lipschitz continuous with Lipschitz coefficient $\frac{L_c}{1-L_P}$.
\end{enumerate}
\end{proposition}
\begin{proof}
The fact that the quotient space $\ALP V/\sim$ is a Banach space is established in \cite[Theorem 6.26, p. 151]{almudevar2014approximate}. The operator $T$ is a span norm contraction, which is established in \cite[Theorem 6.28, p. 154]{almudevar2014approximate}. The only thing that needs to be established is that the value function lies in the class of continuous functions. Since our action space is compact, the continuity of the sequence of value functions generated by the value iteration algorithm follows from \cite[Lemma 4.3]{maitra1968discounted} or from Berge's maximum theorem \cite[Theorem 17.31, p. 570]{ali2006}. The Lipschitz continuity of $v^*$ follows from \cite[Theorem 4.2(c), p.13]{hinderer2005lipschitz}.
\end{proof}

In what follows, we drop the notation $[v]$ to denote an element of $\ALP V/\sim$ and instead just use $v$, with the understanding that any representative element $\tilde v\in[v]$ can be taken to perform the operation involved. Due to Proposition \ref{prop:avgmdp} above, we conclude that Assumption \ref{assm:hatt}(i)-(iii) holds. Using Proposition \ref{prop:avgmdp}(c) along with Proposition \ref{prop:empiricalmdp} allows us to conclude that Assumption \ref{assm:hatt}(v) holds. Assumption \ref{assm:hatt}(vi) holds since 
\beqq{& \sp(\hat T^n_k(v^*), T(v^*)) \leq \sp(\hat T^n_k(v^*), \hat T^n_k(0))  \\
&+ \sp(\hat T^n_k(0), T(0)) + \sp(\sp(T(0), T(v^*)) \\
& \leq 2\:\sp(v^*),}
where we used Theorem 13.1 of \cite[p. 281]{almudevar2014approximate}.

We next show that Assumption \ref{assm:hatt}(iv) also holds by establishing the following property of $\hat G^n_k$ in the following lemma.
\begin{lemma}
Suppose that Assumption \ref{assm:isolated} holds for the average cost MDP. For any $v_1,v_2\in\ALP V$, then the following holds:
 \beqq{\sp(\hat G^n_k(v_1),\hat G^n_k(v_2))\leq \hat \alpha^n_k \sp(v_1,v_2).}
 where $\alpha^n_k$ satisfies the following condition: For any $\delta\in(0,1-\alpha)$, we have 
 \beq{\lim_{n\rightarrow\infty}\pr{\hat\alpha^n_k>1-\delta} = 0.\label{eqn:hatalphaavg}}
 Further, $\hat \alpha^n_k\leq 1$ almost surely.
\end{lemma}
\begin{proof}
It is clear that $\hat\alpha^n_k$ satisfies
\beqq{\hat\alpha^n_k \leq 1 - \inf_{(s,a)} \hat P^n_k(s_0|s,a).}
The fact that $\hat\alpha^n_k\leq 1$ almost surely follows from Theorem 4.22 in \cite[p. 81]{almudevar2014approximate}. The statement will be proved if we show that
\beqq{\lim_{n\rightarrow\infty}\pr{\inf_{(s,a)} \hat P^n_k(s_0|s,a)>\delta} = 0.}
Let $\ALP K\subset\ALP S\times\ALP A$ be the countable dense subset of $\ALP S\times\ALP A$ and define $\ALP Z_0$ to be the set
\beqq{\ALP Z_0 = \bigcap_{(s,a)\in\ALP K}\Big\{z\in\ALP Z: g(s,a,z) = s_0\Big\}. }
Then, $\ALP Z_0$ is measurable and $\pr{\ALP Z_0} = 1-\alpha$ due to Assumption \ref{assm:isolated} (ii) and (iv). This implies that 
$\inf_{(s,a)\in\ALP S\times\ALP A} \hat P^n_k(s_0|s,a) = \hat P^n_k(\ALP Z_0)$. Using Hoeffding inequality, we conclude that
\beqq{& \pr{ \hat P^n_k(\ALP Z_0) -(1- \alpha) <\alpha+\delta-1}  \\
&\leq \exp(-n(\alpha+\delta-1)^2).}
This readily yields the \eqref{eqn:hatalphaavg}, which establishes the lemma.
\end{proof}

\section{Proofs of the Main Results}\label{sec:proofs}
\begin{figure*}[!t]
\centering
\scalebox{0.65}{\begin{tikzpicture}
\node[state, fill = magenta!30!white] (0) {$\eta$};
\node[state,right=of 0, fill = blue!30!white] (1) {$\eta+1$};
\node[state,right=of 1, fill = blue!30!white] (2) {$\eta+2$};
\node[draw=none,right=of 2] (2-g) {...};
\node[state,right=of {2-g},text depth=0pt, fill = blue!30!white] (w) {$\eta+\beta$};
\node[state,right=of {w},text depth=0pt, fill = green!30!white] (wp1) {$\eta+\beta+1$};
\node[draw=none,right=of wp1] (3-g) {...};
\node[state,right=of {3-g},text depth=0pt, fill = green!30!white] (2w) {$\eta+2\beta$};
\node[draw=none,right=of 2w] (4-g) {...};
\node[state,right=of {4-g},text depth=0pt, fill = red!30!white] (3w) {$\eta+3\beta$};

\draw[ >=latex,
    auto=right, loop above/.style={out=75,in=105,loop}, every loop]
     (w)    edge                node {$p$} (2-g)
     (1)    edge[bend left]     node[above] {$1-p$} (wp1)
     (2-g)  edge                node[above] {$p$}   (2)
     (2-g) edge[bend left]      node {$1-p$} (3-g)
     (2)   edge[bend left]  node [above]{$1-p$} (3-g)
     (w) edge[bend right]      node {$1-p$} (2w)
     (3-g)    edge[bend left] node[above] {$1-p$} (4-g)
     (2w)    edge[bend right] node[above] {$1-p$} (3w)
     (wp1)    edge[bend left] node[above] {$1-p$} (4-g)
     (0)    edge[bend right] node[below] {$1-p$} (w)
     
     (wp1)   edge             node {$p$}   (w)
     (2)   edge             node {$p$}   (1)
     (1)   edge             node {$p$}   (0)
     (0)   edge[loop above] node {$p$}   (0);
\end{tikzpicture}}
\caption{\label{fig:mcy} An illustration of the Markov chain $Y^n_k$, where $\eta = \left\lceil\frac{2}{\delta}\right\rceil $.}
\hrulefill
\end{figure*}

We now turn our attention to proving Theorem \ref{thm:hatt}. Define the error $E^n_k:=\rho\Big(\hat X^n_k,x^\star\Big)$. We focus here on the error process $(E^n_k)_{k\in\Na}$ and find an upper bound on the probability of large error for a given $n\in\Na$.


\subsection{Proof Technique}
Given two random variables $X$ and $Y$ defined on the same probability space, $X$ stochastically dominates $Y$ if $\pr{X>q} \geq \pr{Y>q}$ for all $q\in\Re$ \cite{shaked2007}. To prove Theorem \ref{thm:hatt}, we show that the error $E^n_k$ is stochastically dominated by a scaled version of a Markov chain constructed over the space of natural numbers. We prove that if $n$ is sufficiently large, then the dominating Markov chain has an invariant distribution, which allows us to compute an upper bound on the probability of asymptotic error to be greater than $\kappa$. Through this approach, we can also compute the rate of convergence as $n\rightarrow\infty$. We now introduce some notation and proof technique in greater details below.

The error evolution can be written as 
\beqq{E^n_k & = \rho(\hat X^n_k , x^\star ) = \rho(\hat T^n_{k-1}(\hat X^n_{k-1}) , x^\star ),\\
&  \leq  \rho(\hat T^n_{k-1}(\hat X^n_{k-1}),\hat T^n_{k-1}(x^\star))+\rho(\hat T^n_{k-1}(x^\star),T(x^\star))\\
& \leq \hat \alpha^n_{k-1} E^n_{k-1} + W_{k-1}^n }
where $\hat \alpha^n_{k-1}$ denotes the contraction coefficient of $\hat T^n_{k-1}$ and 
\beqq{W^n_{k-1}:=\rho(\hat T^n_{k-1}(x^\star),T(x^\star)) +\hat \zeta^n_k.}
We note here that by Assumption \ref{assm:hatt}(vi), $W^n_k\leq \bar w$ almost surely for any $n\in\Na$ and $k\in\Na$.
\begin{remark}\label{rem:alphaw}
Note that $(\hat \alpha^n_k, W^n_k)$ are functions of $\hat T^n_k$. Since $\hat T^n_k$ is not correlated with $\hat T^n_j$ for any $j\neq k$, we conclude that $\big(\hat \alpha^n_k, W^n_k\big)_{k\in\Na}$ is a sequence of i.i.d. tuple of random variables. However, for every $k\in\Na$, $\hat\alpha^n_k$ is correlated with $W^n_k$.
\end{remark}

We now focus on devising a Markov chain over the space of natural numbers that dominates the error process. Fix $\kappa>0$ and pick $\epsilon\in\big(0,\frac{\kappa(1-\alpha)}{2}\big],\delta\in(0,1-\alpha)$ such that $\left\lceil\frac{2}{\delta}\right\rceil \leq \frac{\kappa}{\epsilon}$. For this choice of $\epsilon$ and $\delta$, pick $\eta^n_{\epsilon,\delta} :=\left\lceil\frac{2}{\delta}\right\rceil$ and $p^n_{\epsilon,\delta}$ such that $p^n_{\epsilon,\delta}\leq \pr{\hat\alpha^n_k\leq 1-\delta,W^n_k \leq \epsilon}$. Recall that $\beta:=\left\lceil\frac{\bar w}{\epsilon}\right\rceil$.  

We now define a Markov chain $(Y^n_k)_{k\in\Na}$ on the set of natural numbers as follows: Let $Y^n_1 =\lceil E^n_1/\epsilon\rceil$, and assume that the chain evolves as
\beq{\label{eqn:ynk1} Y^n_{k+1} = 
\left\{\begin{array}{ll} 
\eta^n_{\epsilon,\delta}        & \text{ with probability } p^n_{\epsilon,\delta} \text{ if } Y^n_k = \eta^n_{\epsilon,\delta}\\ 
Y^n_k - 1                       & \text{ with probability } p^n_{\epsilon,\delta} \text{ if } \\
& Y^n_k \geq  \eta^n_{\epsilon,\delta}+1\\
Y^n_k+\beta & \text{ with probability } 1-p^n_{\epsilon,\delta}\end{array}\right.}
\new{Figure \ref{fig:mcy} depicts the communication structure of this Markov chain. }

If $n$ is sufficiently large such that $p^n_{\epsilon,\delta}$ satisfies $p^n_{\epsilon,\delta}> \beta/(\beta+1)$, then we show that $Y^n_k$ admits a unique invariant distribution since it is irreducible and has a negative drift. Further, we prove that at every step of the iteration $k$, $\epsilon Y^n_k$ stochastically dominates the error random variable $E^n_k$, that is, for any real number $q\in[0,\infty)$ and $k\in\Na$,
\beqq{\pr{\epsilon Y^n_k> q} \geq \pr{E^n_k> q}.} 
This yields for every $k\in\Na$, we get
\beqq{\pr{E^n_k>\kappa} \leq \pr{\epsilon Y^n_k> \kappa}\leq \pr{Y^n_k> \eta^n_{\epsilon,\delta}}.}
Let $\pi^n$ denote the invariant distribution of the Markov chain $(Y^n_k)_{k\in\Na}$. Then, the above inequality implies
\beqq{ \limsup_{k\rightarrow\infty} \pr{E^n_k>\kappa}\leq \lim_{k\rightarrow\infty} \pr{Y^n_k> \eta^n_{\epsilon,\delta}} = 1-\pi^n(\eta^n_{\epsilon,\delta}).}
Further, as $n$ grows, we show that the invariant distribution at $\eta^n_{\epsilon,\delta}$, $\pi^n(\eta^n_{\epsilon,\delta})$, converges to 1, thereby proving the convergence of the error process $E^n_k$ to 0 in probability as $k\rightarrow\infty$ and $n\rightarrow\infty$.

\subsection{Dominating the Error with the Markov Chain \texorpdfstring{$(Y^n_k)$}{}}
Recall that the error evolves as $E^n_{k+1} \leq \hat\alpha^n_k E^n_k+W^n_k$, where $(\hat \alpha^n_k,W^n_k)_{k\in\Na}$ is an i.i.d. sequence of random variables.
\begin{proposition}\label{PROP:DOMINATE}
Let $n$ be large such that $\pr{\hat\alpha^n_k\leq 1-\delta,W^n_k \leq \epsilon}>1/2$. Pick $p^n_{\epsilon,\delta}\in(0.5,1)$ such that $p^n_{\epsilon,\delta}\leq \pr{\hat\alpha^n_k\leq 1-\delta,W^n_k \leq \epsilon}$ and consider the Markov chain $(Y^n_k)_{k\in\Na}$ constructed in \eqref{eqn:ynk1}. If $\epsilon Y^n_1\geq E^n_1$, then at every iteration $k$, $\epsilon Y^n_k$ stochastically dominates $E^n_k$. In other words, for any $k\in\Na$ and any real number $q\in[0,\infty)$,
\beqq{\pr{\epsilon Y^n_k> q} \geq \pr{E^n_k> q}.} 
\end{proposition}
\begin{proof}
See Appendix \ref{app:dominate}.
\end{proof}

In the light of the proposition above, we need to identify a lower bound on the joint distribution $\pr{\hat\alpha^n_k\leq 1-\delta,W^n_k \leq \epsilon}$ that can be used to determine $p^n_{\epsilon,\delta}$. We obtain a lower bound by using the Fr\'echet-Hoeffding theorem \cite[Theorem 3.1.1]{rachev1998i}:
For any $a_1,a_2\in[0,\infty)$, we have
\beq{\pr{\hat\alpha^n_k\leq a_1,W^n_k \leq a_2} \geq \pr{\hat\alpha^n_k\leq a_1}+\pr{W^n_k \leq a_2}-1.\label{eqn:frechethoeffding}}

Suppose that we know the upper bounds $\varphi_1(n,\delta)$ and $\varphi_2(n,\epsilon)$ on the probability $\pr{\hat\alpha^n_k> 1-\delta}$ and $\pr{W^n_k> \epsilon}$, respectively:
\beqq{\pr{\hat\alpha^n_k> 1-\delta}\leq \varphi_1(n,\delta),\qquad \pr{W^n_k >\epsilon}\leq \varphi_2(n,\epsilon).}
Typically, such bounds can be obtained using a concentration of measures result such as the Hoeffding inequality or the empirical process theory \cite{pollard1984convergence,van1996weak,gyorfi2002principles}. We let $p^n_{\epsilon,\delta}$ be defined as $p^n_{\epsilon,\delta} = 1-\varphi_1(n,\delta)-\varphi_2(n,\epsilon)$. This, together with \eqref{eqn:frechethoeffding}, implies
\beqq{& \pr{\hat\alpha^n_k\leq 1-\delta,W^n_k \leq \epsilon} \\
& \geq \pr{\hat\alpha^n_k \leq 1-\delta}+\pr{W^n_k \leq \epsilon}-1\\
& \geq 1-\varphi_1(n,\delta)-\varphi_2(n,\epsilon) = p^n_{\epsilon,\delta} .}
We make the following observation.

\begin{lemma}\label{lem:pn}
If Assumption \ref{assm:hatt} holds, then for any $\epsilon>0$ and $\delta\in (0,1-\alpha)$, $\lim_{n\rightarrow\infty}p^n_{\epsilon,\delta} = 1$.
\end{lemma}
\begin{proof}
The proof essentially follows from Assumption \ref{assm:hatt}. Assumption \ref{assm:hatt}(iv) implies that $\varphi_1(n,\delta)\rightarrow 0$ as $n\rightarrow\infty$ for any $\delta\in(0,1-\alpha)$. Assumption \ref{assm:hatt}(v) implies that $\varphi_2(n,\epsilon)\rightarrow 0$ as $n\rightarrow\infty$ for any $\epsilon>0$. The proof of the lemma is complete.
\end{proof}

According to the Lemma above, by picking $n$ sufficiently large, $p^n_{\epsilon,\delta}$ can be made as close to 1 as possible. As we show next, for $p^n_{\epsilon,\delta}$ sufficiently close to 1, the Markov chain $(Y^n_k)_{k\in\Na}$ admits an invariant distribution.

\begin{proposition}\label{PROP:INVARIANT}
If $p^n_{\epsilon,\delta}>\beta/(\beta+1)$, then the Markov chain $(Y^n_k)_{k\in\Na}$ admits a unique invariant distribution $\pi^n$. If $p^n_{\epsilon,\delta}>2\beta/(2\beta+1)$, then 
\beqq{\pi^n\big(\eta^n_{\epsilon,\delta}\big)\geq \frac{2(p^n_{\epsilon,\delta})^\beta-1}{\big(p^n_{\epsilon,\delta}\big)^\beta}.}
\end{proposition}
\begin{proof}
We first note that $(Y^n_k)_{k\in\Na}$ is an irreducible Markov chain, in which all integers greater than or equal to $\eta^n_{\epsilon,\delta}$ are accessible. Consider the Lyapunov function as $V(y) = y$ for all $y\in\Na$, $y\geq \eta^n_{\epsilon,\delta}$. Then, $\ex{V(Y^n_2)|Y^n_1 = y} - V(y) <0$ for all $y\geq \eta^n_{\epsilon,\delta} +1$. The uniqueness of the invariant distribution follows from \cite[Theorem 7.5.3, p.153]{douc2018markov}. See Appendix \ref{app:invariant} for a derivation on the lower bound on $\pi^n\big(\eta^n_{\epsilon,\delta}\big)$.
\end{proof}

\subsection{Proof of Theorem \ref{thm:hatt}}
We now prove Theorem \ref{thm:hatt} using Propositions \ref{PROP:DOMINATE} and \ref{PROP:INVARIANT} as follows. For $n$ sufficiently large (so that $p^n_{\epsilon,\delta}>2\beta/(2\beta+1)$), we can use Proposition \ref{PROP:DOMINATE} to conclude that for every $k\in\Na$,
\beqq{\pr{E^n_k\geq \kappa} \leq \pr{\epsilon Y^n_k\geq \kappa} \leq \pr{\epsilon Y^n_k> \epsilon\eta^n_{\epsilon,\delta}}.}
From Proposition \ref{PROP:INVARIANT}, we conclude that
\beqq{\lim_{k\rightarrow\infty}\pr{Y^n_k> \eta^n_{\epsilon,\delta}} & = 1-\pi^n(\eta^n_{\epsilon,\delta}) \leq \frac{\big(1-(p^n_{\epsilon,\delta})^\beta\big)}{(p^n_{\epsilon,\delta})^\beta},}
Consequently, we have
\beq{\label{eqn:pnw}\underset{k\rightarrow\infty}{\lim\sup}\;\;\pr{E^n_k\geq \kappa} \leq \frac{(1-\big(p^n_{\epsilon,\delta})^\beta\big)}{(p^n_{\epsilon,\delta})^\beta}.}
The proof of the theorem is complete.

\subsection{Proof of Theorem \ref{thm:hath}}\label{sub:hath}
It is easy to observe that Assumption \ref{assm:hath} immediately implies Assumption \ref{assm:hatt}(iii) and (iv). We only need to show that Assumption \ref{assm:hatt}(v) and (vi) holds.

Let us consider the case of $L=2$ for simplicity, and the general case can then be easily deduced. For any $x^\star\in\ALP X$, define $x^\star_1 := x^\star, x^\star_2 := T_1(x^\star_1)$, we have
\beqq{& \rho(\hat T^n_k(x^\star), T(x^\star)) = \rho_3\Big(\hat H^n_{k,2}\big(\hat H^n_{k,1}(x^\star_1)\big), T_2\big(T_1(x^\star_1)\big)\Big),\\
& \leq \rho_3\Big(\hat H^n_{k,2}\big(\hat H^n_{k,1}(x^\star_1)\big), \hat H^n_{k,2}\big(T_1(x^\star_1)\big)\Big) \\
& + \rho_3\Big(\hat H^n_{k,2}\big(T_1(x^\star_1)\big), T_2\big(T_1(x^\star_1)\big)\Big),\\
& \leq \hat \alpha^n_{k,2}\rho_2\Big(\hat H^n_{k,1}(x^\star_1),T_1(x^\star_1)\Big) +\hat \zeta^n_k\\
& + \rho_3\Big(\hat H^n_{k,2}\big(x^\star_2\big), T_2\big(x^\star_2\big)\Big).}
We now extend the same argument to the general case. 
Using triangle inequality and Assumption \ref{assm:hath}(ii), we get
\beq{& \rho(\hat T^n_k(x^\star), T(x^\star)) \nonumber\\
&= \rho_{L+1}(\hat H^n_{k,L}\circ\ldots\circ \hat H^n_{k,1}(x^\star),T_L\circ \ldots\circ T_1(x^\star)),\nonumber\\
& \leq \sum_{l=1}^L \left(\prod_{m=l+1}^L\hat\alpha^n_{k,m}\right)\rho_{l+1}\Big(\hat H^n_{k,l}(x^\star_l),T_l(x^\star_l)\Big) +\hat\zeta^n_k,\label{eqn:step1}\\
& \leq \sum_{l=1}^L\rho_{l+1}\Big(\hat H^n_{k,l}(x^\star_l),T_l(x^\star_l)\Big) +\hat\zeta^n_k,\label{eqn:step2}}
where we take $\prod_{m=L+1}^L\hat\alpha^n_{k,m} = 1$. The last inequality immediately yields for any $\epsilon>0$,
\beqq{& \pr{\rho\Big(\hat T^n_k(x^\star),T(x^\star)\Big)\geq \epsilon} \\
& \leq \sum_{l=1}^{L-1} \pr{\rho_{l+1}\Big(\hat H^n_{k,l}(x^\star_l),T_l(x^\star_l)\Big) \geq \frac{\epsilon}{L}}\\
& +\pr{\rho_{L+1}\Big(\hat H^n_{k,L}(x^\star_L),T_L(x^\star_L)\Big)+\hat \zeta^n_k \geq \frac{\epsilon}{L}}.}
This implies that Assumption \ref{assm:hatt}(v) holds. To see that Assumption \ref{assm:hatt}(vi) holds, we only need to take $\bar w = \sum_{l=1}^L \bar w_l$. The proof is hence complete.

\subsection{Proof of Theorem \ref{thm:samplecomplex}} \label{sub:samplecomplex}
This proof is inspired from \cite{haskell2016}. For this case, a dominating Markov chain is defined as
\beqq{ Y^n_{k+1} = 
\left\{\begin{array}{ll} 
\eta        & \text{ with probability } p \text{ if } Y^n_k = \eta\\ 
Y^n_k - 1                       & \text{ with probability } p \text{ if } Y^n_k \geq  \eta+1\\
d & \text{ with probability } 1-p\end{array}\right.}
It can be shown that for any $k\in\Na$, we have
\beqq{\pr{\rho(\hat X^n_k,x^\star)>\kappa} \leq \pr{\epsilon Y^n_k > \kappa}.}
Lemma 5.1 and Proposition 5.3 in \cite{haskell2016} shows that for all $k\geq K_{\kappa,\xi}$ and $n\geq N_{\kappa,\xi}$, $ \pr{\epsilon Y^n_k > \kappa}<\xi$. The proof is hence complete.

\section{Conclusions}\label{sec:conclusion}

In this paper, we have introduced a novel approach to deduce the stability of iterated random operators. The approach presented here provides an alternative to the Foster-Lyapunov method, the martingale method, and the ODE method for use in infinite dimensional spaces. We refer to this approach as the \textit{probabilistic contraction analysis}. We hope that this paper can accelerate the development of new recursive stochastic algorithms along with their convergence guarantees.

In our future work, we will also consider other characterizations of probabilistic fixed points (e.g., a mean square version), and explore application to other reinforcement learning problems.

\appendices

\section{Proof of Proposition \ref{PROP:DOMINATE}}\label{app:dominate}
Let $\ALP E_k$ denote the event $E^n_k\leq \epsilon\eta^n_{\epsilon,\delta}$ and $\ALP F_k$ denote the event $\{W^n_k\leq \epsilon,\hat\alpha^n_k\leq 1-\delta\}$.

\begin{lemma}\label{lem:enk1}
 We have
\beqq{[E^n_{k+1}\Big| \ALP E_k,\ALP F_k] &< \epsilon\eta^n_{\epsilon,\delta} \quad \text{almost surely},\\
[E^n_{k+1} - E^n_k\Big| \ALP E_k^\complement,\ALP F_k] &\leq -\epsilon \quad \text{almost surely},\\
E^n_{k+1} - E^n_k &\leq \bar w \quad \text{almost surely}.
}
\end{lemma}
\begin{proof}
For establishing the first inequality, we readily obtain
\beqq{[E^n_{k+1}\Big| \ALP E_k,\ALP F_k] & = \hat\alpha^n_k E^n_k+ W^n_k \leq (1-\delta)(\frac{2\epsilon}{\delta}+\epsilon) +\epsilon \\
& = \epsilon\left( \frac{2}{\delta} - \delta\right) <\epsilon \eta^n_{\epsilon,\delta}.}
To establish the second inequality, note that
\beq{\label{eqn:enk1}E^n_{k+1} - E^n_k \leq \hat\alpha^n_kE^n_k +W^n_k - E^n_k = W^n_k - (1-\hat\alpha^n_k)E^n_k.}
Under the event $\ALP E_k^\complement$ and $\ALP F_k$, we have
\beqq{[E^n_{k+1} - E^n_k|\ALP E_k^\complement,\ALP F_k] \leq \epsilon - \delta E^n_k\leq   \epsilon - \delta \times 2\epsilon/\delta = -\epsilon.}
Moreover, from \eqref{eqn:enk1} and due to the fact that $\hat\alpha^n_k\leq 1$ and $E^n_k\geq 0$ almost surely, we readily conclude that $E^n_{k+1} - E^n_k \leq \bar w$ almost surely.
\end{proof}

We now construct another Markov chain $(Z^n_k)_{k\in\Na}$ on $(\Omega,\ALP F,\mathbb{P})$ as follows: Let $Z^n_1 =\lceil E^n_1/\epsilon\rceil$, and the chain evolves as
\beq{\label{eqn:znk1} Z^n_{k+1} = \left\{\begin{array}{ll}
\eta^n_{\epsilon,\delta} & \text{ if } Z^n_k = \eta^n_{\epsilon,\delta} \\ 
Z^n_k - 1 & \text{ if } W^n_k\leq \epsilon, \hat\alpha^n_k\leq 1-\delta,\\
& \text{ and } Z^n_k \geq  \eta^n_{\epsilon,\delta}+1
\\ Z^n_k+\left\lceil\frac{\bar w}{\epsilon}\right\rceil & \text{ if } W^n_k\geq \epsilon \text{ or } \hat\alpha^n_k\geq 1-\delta.
\end{array}\right.}
By construction and Lemma \ref{lem:enk1}, we have $E^n_k\leq \epsilon Z^n_k$ almost surely, which can be shown via induction. The statement is obviously true for $k=1$. Assume that the statement is true for some $k$. Then, two cases can occur: 
\begin{enumerate}
 \item If $E^n_k\leq \epsilon\eta^n_{\epsilon,\delta}$, which implies that $E^n_{k+1}$ will be less than $ \epsilon\eta^n_{\epsilon,\delta}$ in the event $\ALP F_k$ by Lemma \ref{lem:enk1}. By construction, $Z^n_{k+1}\geq \eta^n_{\epsilon,\delta}$, and thus, $E^n_{k+1} \leq \epsilon Z^n_{k+1}$.
 \item On the other hand, if $E^n_k> \epsilon\eta^n_{\epsilon,\delta}$, then by Lemma \ref{lem:enk1}, $E^n_{k+1}$ is less than or equal to $E^n_k - \epsilon$ under the event in $\ALP F_k$. Since $Z^n_k\geq E^n_k/\epsilon$, we have $Z^n_k \geq \eta^n_{\epsilon,\delta}+1$ and $Z^n_{k+1} = Z^n_k - 1$ under the event $\ALP F_k$.
 \item It is clear from Lemma \ref{lem:enk1} that $E^n_{k+1}\leq E^n_k+\bar w\leq Z^n_k+\bar w$ in the event $\ALP F_k^\complement$, which implies $E^n_{k+1} \leq Z^n_{k+1}$ under the event $\ALP F_k^\complement$.
\end{enumerate}
As a result of the assertion above, $Z^n_k$ stochastically dominates $E^n_k$. We now have the following lemma.
\begin{lemma}
Let $T_Y$ and $T_Z$ denote the transition kernel of $(Y^n_k)$ and $(Z^n_k)$, respectively. Then, for any $q\in \Na$, $q\geq \eta^n_{\epsilon,\delta}$, $T_Y(q,\cdot)$ stochastically dominates $T_Z(q,\cdot)$. Consequently, $Y^n_k$ stochastically dominates $Z^n_k$.
\end{lemma}
\begin{proof}
Since $p^n_{\epsilon,\delta}\leq \pr{W^n_k\leq \epsilon, \hat\alpha^n_k\leq 1-\delta}$, $T_Y(q,\cdot)$ stochastically dominates $T_Z(q,\cdot)$ for all $q\in\Na$, $q\geq \eta^n_{\epsilon,\delta}$. The fact that $Y^n_k$ stochastically dominates $Z^n_k$ then follows from Proposition 1 in \cite[p. 901]{kamae1977stochastic} (see also the remarks following this proposition).
\end{proof}

Consequently, we proved that $\epsilon Y^n_{k+1}$ stochastically dominates $\epsilon Z^n_{k+1}$, which in turn stochastically dominates $E^n_{k+1}$. The induction step is complete and we arrive at the result.

\section{Proof of Proposition \ref{PROP:INVARIANT}}\label{app:invariant}
Let $(\tilde\Omega,\tilde{\ALP F},\tilde{\mathbb{P}})$ be a standard probability space. On this probability space, we define two different Markov chains: $(P_k)$ and $(Q_k)$. Pick $p\in(0,1]$. Markov chains $P_k$ and $Q_k$, $k\in\Na$, evolves as
\beqq{P_{k+1} &= \left\{ \begin{array}{ll} 0 & \text{with probability } p \text{ if } P_k = 0\\ 
P_k - 1 & \text{ with probability } p \text{ if } P_k \geq  1\\ 
P_k+\beta & \text{with probability } 1-p\end{array}\right.,}
\beqq{Q_{k+1} &= \left\{ \begin{array}{ll} 0 & \text{with probability } p \text{ if } Q_k = 0\\ 
Q_k - 1 & \text{with probability } p \text{ if } Q_k \geq  1\\ 
\beta \Big(\lceil Q_k/\beta\rceil+1\Big) & \text{with probability } 1-p\end{array}\right.}
\begin{figure*}[!t]
\centering
\scalebox{0.74}{\begin{tikzpicture}
\node[state, fill = magenta!30!white] (0) {$0$};
\node[state,right=of 0, fill = blue!60!white] (1) {$1$};
\node[state,right=of 1, fill = blue!60!white] (2) {$2$};
\node[draw=none,right=of 2] (2-g) {...};
\node[state,right=of {2-g},text depth=0pt, fill = blue!60!white] (w) {$w$};
\node[state,right=of {w},text depth=0pt, fill = green!50!white] (wp1) {$w+1$};
\node[draw=none,right=of wp1] (3-g) {...};
\node[state,right=of {3-g},text depth=0pt, fill = green!50!white] (2w) {$2w$};
\node[draw=none,right=of 2w] (4-g) {...};
\node[state,right=of {4-g},text depth=0pt, fill = red!50!white] (3w) {$3w$};

\draw[ >=latex,
    auto=right, loop above/.style={out=75,in=105,loop}, every loop]
     (w)    edge                node {$p$} (2-g)
     (1)    edge[bend left]     node[above] {$1-p$} (wp1)
     (2-g)  edge                node[above] {$p$}   (2)
     (2-g) edge[bend left]      node {$1-p$} (3-g)
     (2)   edge[bend left]  node [above]{$1-p$} (3-g)
     (w) edge[bend right]      node {$1-p$} (2w)
     (3-g)    edge[bend left] node[above] {$1-p$} (4-g)
     (2w)    edge[bend right] node[above] {$1-p$} (3w)
     (wp1)    edge[bend left] node[above] {$1-p$} (4-g)
     (0)    edge[bend right] node[below] {$1-p$} (w)
     
     (wp1)   edge             node {$p$}   (w)
     (2)   edge             node {$p$}   (1)
     (1)   edge             node {$p$}   (0)
     (0)   edge[loop above] node {$p$}   (0);
\end{tikzpicture}}
\scalebox{0.74}{\begin{tikzpicture}
\node[state, fill = magenta!30!white] (0) {$0$};
\node[state,right=of 0, fill = blue!60!white] (1) {$1$};
\node[state,right=of 1, fill = blue!60!white] (2) {$2$};
\node[draw=none,right=of 2] (2-g) {...};
\node[state,right=of {2-g},text depth=0pt, fill = blue!60!white] (w) {$w$};
\node[state,right=of {w},text depth=0pt, fill = green!50!white] (wp1) {$w+1$};
\node[draw=none,right=of wp1] (3-g) {...};
\node[state,right=of {3-g},text depth=0pt, fill = green!50!white] (2w) {$2w$};
\node[draw=none,right=of 2w] (4-g) {...};
\node[state,right=of {4-g},text depth=0pt, fill = red!50!white] (3w) {$3w$};

\draw[ >=latex,
    auto=right, loop above/.style={out=75,in=105,loop}, every loop]
     (w)    edge                node {$p$} (2-g)
     (1)    edge[bend left]     node[above] {$1-p$} (2w)
     (2-g)  edge                node[above] {$p$}   (2)
     (2-g) edge[bend left]      node {$1-p$} (2w)
     (2)   edge[bend left]  node [above]{$1-p$} (2w)
     (3-g)    edge[bend right] node[below] {$1-p$} (3w)
     (2w)    edge[bend right] node[above] {$1-p$} (3w)
     (wp1)    edge[bend right] node[below] {$1-p$} (3w)
     (0)    edge[bend right] node[below] {$1-p$} (w)
     (w) edge[bend right]      node {$1-p$} (2w)

     (wp1)   edge             node {$p$}   (w)
     (2)   edge             node {$p$}   (1)
     (1)   edge             node {$p$}   (0)
     (0)   edge[loop above] node {$p$}   (0);
\end{tikzpicture}}
\caption{An illustration of the communication structure of the Markov chains $P_k$ (above) and $Q_k$ (below).}
\hrulefill
\end{figure*}

Both Markov chains thus constructed are supported over the space of non-negative integers, are irreducible, and all non-negative integers are accessible. It is also clear that if $p>2\beta/(2\beta+1)$, then by Theorem 7.5.3 of \cite{douc2018markov}, both Markov chains admit a unique invariant distribution. We next have the following claim:
\begin{claim}
If $P_1 = Q_1$, then $Q_k$ stochastically dominates $P_k$ for every $k\in\Na$. 
\end{claim}
\begin{proof}
We show that along every sample path, $Q_k(\tilde \omega)\geq P_k(\tilde\omega)$. Suppose that $P_k = Q_k = q$ for some $q\in\Na$. Then, for any $\tilde\omega\in\tilde\Omega$, either $P_{k+1} = Q_{k+1} = \max\{0,q-1\}$, or  
\beqq{Q_{k+1} & = \beta (\lceil Q_k/\beta\rceil+1) \\
& \geq \beta (Q_k/\beta+1) = Q_k+\beta = P_k+\beta.}
Thus, $Q_{k+1}\geq P_{k+1}$. The result then holds from Theorem 1.A.6 in \cite{shaked2007}.
\end{proof}
As a result of the claim above, if both Markov chains $(P_k)$ and $(Q_k)$ admit invariant distributions $\pi^P$ and $\pi^Q$, respectively, then $\pi^P(0)\geq \pi^Q(0)$. We next identify certain sufficient conditions under which the two Markov chains admit invariant distributions.
\begin{theorem}\label{thm:pq}
The following holds true:
\begin{enumerate}
\item If $p>\beta/(\beta+1)$, then the Markov chain $(P_k)$ has an invariant distribution $(\pi^P(n))_{n=0}^\infty$.
\item If $p>2\beta/(2\beta+1)$, then $Q_k$ has an invariant distribution $(\pi^Q(n))_{n=0}^\infty$.
\end{enumerate}
\end{theorem}
\begin{proof}
We show that both Markov chains are weak Feller chains since they are defined over a countable state space. 

Let $\gamma_P = (\beta+1)p-\beta>0$. Consider $V_P(i) = (i+1)/\gamma_P$ and compact set $C_P = \{0\}$. Then, given $P_k = i\geq 1$, we have
\beqq{& \ex{V_P(P_{k+1})\big|P_k} - V_P(P_k) \\
& = \frac{pP_k+(1-p)(P_k+\beta+1) -(P_k+1)}{\gamma_P}= -1.}
For $P_k = 0$, we have
\beqq{& \ex{V_P(P_{k+1})\big|P_k = 0} - V_P(0) \\
& = \frac{p+(1-p)(\beta+1) -1}{\gamma_P}= -1+\frac{p}{\gamma_P}.}
Thus, by Theorem 12.3.4 of \cite{meyn2012}, an invariant probability distribution $\pi^P$ for the Markov chain $(P_k)$ exists.

Let $\gamma_Q = (2\beta+1)p-2\beta>0$. Consider $V_Q(i) = (i+1)/\gamma_Q$ and compact set $C_Q = \{0\}$. Then, given $Q_k = i\geq 1$, we have
\beqq{& \ex{V_Q(Q_{k+1})\big|Q_k} - V_Q(Q_k) \\
& = \frac{pQ_k+(1-p)(\beta\lceil \frac{Q_k}{w}\rceil +\beta+1) -(Q_k+1)}{\gamma_Q}\\
& \leq  \frac{pQ_k+(1-p)(Q_k+2\beta+1) -(Q_k+1)}{\gamma_Q}\\
& = -1.}
For $Q_k = 0$, we have
\beqq{& \ex{V_Q(Q_{k+1})\big|Q_k = 0} - V_Q(0) \\
& = \frac{p+(1-p)(\beta+1) -1}{\gamma_Q}\\
& \leq  \frac{p+(1-p)(2\beta+1) -1}{\gamma_Q} = -1+\frac{p}{\gamma_Q}.}
Again, we invoke Theorem 12.3.4 of \cite{meyn2012} to conclude the existence of an invariant probability distribution $\pi^Q$ for the Markov chain $(Q_k)$.
\end{proof}

We characterize the invariant distribution $\pi^Q$ in the following claim.
\begin{claim}\label{cl:qinv}
Assume that $p>2\beta/(2\beta+1)$. Then, the invariant distribution of the Markov chain $(Q_k)$ satisfies $\pi^Q(0) = \frac{2p^\beta-1}{p^\beta}$ and for all $i\in\{1,\ldots,\beta\}$ and $k\geq 0$,
\beqq{\pi^Q(k\beta+i)=  \pi^Q(0)\frac{(1-p)}{p^{k\beta+i}}(1-p^\beta)^{k}.}
\end{claim}
\begin{proof}
See Subsection \ref{sub:qinv} below.
\end{proof}

\begin{corollary}
Assume that $p>2\beta/(2\beta+1)$. Then, an invariant distribution $\pi^P$ exists and $\pi^P(0)\geq \pi^Q(0) = \frac{2p^\beta-1}{p^\beta}$.
\end{corollary}
\begin{proof}
Note that $2\beta/(2\beta+1)>\beta/(\beta+1)$. The proof then follows immediately from Theorem \ref{thm:pq} and Claim \ref{cl:qinv}.
\end{proof}

It is now easy to observe that if $p = p^n_{\epsilon,\delta}$, then the evolution of  $Y^n_k$ is the same as that of $P^n_k+\eta^n_{\epsilon,\delta}$. Thus, their invariant distribution is a ``shifted'' version of the other, that is, $\pi^n(i) = \pi^P(i-\eta^n_{\epsilon,\delta})$ for all $i\geq \eta^n_{\epsilon,\delta}$. Thus, $\pi^n(\eta^n_{\epsilon,\delta}) \geq \frac{2(p^n_{\epsilon,\delta})^\beta-1}{(p^n_{\epsilon,\delta})^\beta} $. This completes the proof of Proposition \ref{PROP:INVARIANT}. 

\subsection{Proof of Claim \ref{cl:qinv}}\label{sub:qinv}
The invariant distribution $\pi^Q$ exists by Theorem \ref{thm:pq} above. It must satisfy
\beqq{\pi^Q(0) &= p\pi^Q(0)+p\pi^Q(1)  \implies \pi^Q(1) = \frac{(1-p)}{p}\pi^Q(0),\\
\pi^Q(i) &= p\pi^Q(i+1) \quad\text{ for all }  i\in\{1,\ldots,\beta-1\},}
which implies that $\pi^Q(i) = \frac{(1-p)}{p^i}\pi^Q(0)$ for all $i\in\{1,\ldots,\beta\}$. Thus, the statement holds for $n=0$ and all $i\in\{0,1,\ldots,\beta\}$. For $n= 1$, we have 
\beqq{\pi^Q(\beta) &= (1-p)\pi^Q(0)+p\pi^Q(\beta+1),\\
\pi^Q(\beta+i) &= p\pi^Q(\beta+i+1)\quad \text{ for all } i\in\{1,\ldots,\beta-1\},}
which implies 
\beqq{\pi^Q(\beta+1) &=  \frac{1}{p}\left( \pi^Q(\beta)-(1-p)\pi^Q(0) \right) \\
& = \frac{(1-p)}{p^{\beta+1}}\pi^Q(0)(1-p^{w})\\
\pi^Q(\beta+i) &=  \frac{(1-p)}{p^{\beta+i}}\pi^Q(0)(1-p^\beta) \text{ for all } i\in\{1,\ldots,\beta\}.}
Consequently, the statement holds for $n=1$ as well. We now prove the result for arbitrary $n\geq 2$. Suppose that the result holds for all $m\leq n-1$ and $i\in\{1,\ldots,\beta\}$. Then, we have
\beq{& \pi^Q(m\beta+1)+\ldots+\pi^Q(m\beta+\beta) \nonumber\\
& = \frac{(1-p)}{p^{m\beta+\beta}}\pi^Q(0)(1-p^\beta)^m\left( 1+\ldots+p^{\beta-1}\right),\nonumber\\
& = \frac{1}{p^{(m+1)\beta}}\pi^Q(0)(1-p^\beta)^{m+1}.\label{eqn:pim1}}
Next, we have 
\beqq{& \pi^Q(n\beta) \\
&= (1-p)\Big(\pi^Q((n-2)\beta+1)+\ldots+\pi^Q((n-2)\beta+\beta)\Big)\\
& \quad +p\pi^Q(n\beta+1),\\
& = \frac{(1-p)}{p^{(n-1)\beta}}\pi^Q(0)(1-p^\beta)^{n-1}+p\pi^Q(n\beta+1),\\
& \pi^Q(n\beta+i) = p\pi^Q(n\beta+i+1) \text{ for all } i\in\{1,\ldots,\beta-1\}.}
Using similar approach as for the previous cases, we have
\beqq{& \pi^Q(n\beta+1) \\
& = \frac{1}{p}\left[\pi^Q((n-1)\beta+\beta)- \frac{(1-p)}{p^{(n-1)\beta}}\pi^Q(0)(1-p^\beta)^{n-1} \right]\\
& = \frac{1}{p}\times\frac{(1-p)}{p^{n\beta}}\pi^Q(0)(1-p^\beta)^{n-1} \left(1-p^\beta \right),}
\beqq{ \pi^Q(n\beta+i+1) & = \frac{1}{p}\pi^Q(n\beta+i) \\
& = \frac{(1-p)}{p^{n\beta+i+1}}\pi^Q(0)(1-p^\beta)^{n},}
which holds for all $i\in\{1,\ldots,\beta-1\}$. Thus, the statement is true for $n$. By the principle of mathematical induction, the statement is established.

We can now compute $\pi^Q(0)$ by noting that
\beqq{\pi^Q(0)+\pi^Q(0) \sum_{m=0}^\infty \frac{1}{p^{(m+1)\beta}}(1-p^\beta)^{m+1} = 1,}
where we used \eqref{eqn:pim1}. The above expression yields
\beqq{\pi^Q(0) \left( 1+\frac{\frac{1-p^\beta}{p^\beta}}{1-\frac{1-p^\beta}{p^\beta}} \right) = \pi^Q(0)\frac{p^\beta}{2p^\beta-1} = 1,}
which implies $\pi^Q(0) = \frac{2p^\beta-1}{p^\beta}$. The proof of the claim is complete.

Note that for $\pi^Q(0)$ to be non-negative, we need $p^\beta\geq 0.5$. We show in the following remark that this is indeed true as long as $p\geq 2\beta/(2\beta+1)$.

\begin{remark}\label{rem:pw}
We show that if $p>2\beta/(2\beta+1)$, then $p^\beta>0.6$, which further implies $\pi^Q(0)>0$. To establish this result, we prove that the map $r\mapsto (1+\frac{1}{r})^r$ is monotonically increasing in $r\in(1,\infty)$. Indeed,
\beqq{\frac{d}{dr} \left(1+\frac{1}{r}\right)^r = \left(1+\frac{1}{r}\right)^r\left( \ln\left( 1+\frac{1}{r}\right) - \frac{1}{r+1} \right).}
Now, since for $t\in (1,1+\frac{1}{r})$, $t < \frac{r+1}{r}$ or $\frac{1}{t}> \frac{r}{r+1}$,  we have
\beqq{ \ln\left( 1+\frac{1}{r}\right) - \frac{1}{r+1} = \int_1^{1+\frac{1}{r}} \left(\frac{1}{t}-\frac{r}{r+1}\right) dt >0. }
Thus, $\frac{d}{dr} \left(1+\frac{1}{r}\right)^r>0$, which implies $\left(1+\frac{1}{r}\right)^r$ is monotonically increasing in $r$ in the domain $[1,\infty)$. Thus,  if $p>2\beta/(2\beta+1)$, we have
\beqq{p^\beta & \geq \frac{1}{(1+\frac{1}{2\beta})^\beta} = \frac{1}{\sqrt{\big(1+\frac{1}{2\beta}\big)^{2\beta}}} \\
& \geq \lim_{r\rightarrow \infty}\frac{1}{\sqrt{\big(1+\frac{1}{r}\big)^{r}}} = \frac{1}{\sqrt{e}} \approx 0.606. }
\end{remark}

\bibliographystyle{ieeetr}
\bibliography{edp,prob,probbook,math,ql,rl,sgd,smc,guptajournal}

\begin{IEEEbiography}[{\includegraphics[width=1in,height=1.25in,clip,keepaspectratio]{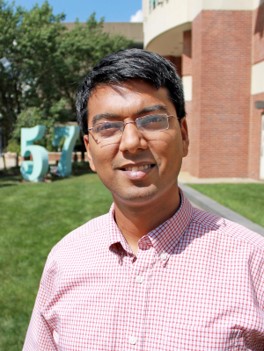}}]{Abhishek Gupta}
Abhishek Gupta is an assistant professor in the ECE department at The Ohio State University. He completed his MS and PhD in Aerospace Engineering from University of Illinois at Urbana-Champaign (UIUC) in 2014, MS in Applied Mathematics from UIUC in 2012, and B.Tech. in Aerospace Engineering from IIT Bombay in 2009. His research interests are in stochastic control theory, probability theory, and game theory with applications to transportation markets, electricity markets, and cybersecurity of control systems.
\end{IEEEbiography}

\begin{IEEEbiography}[{\includegraphics[width=1in,height=1.25in,clip,keepaspectratio]{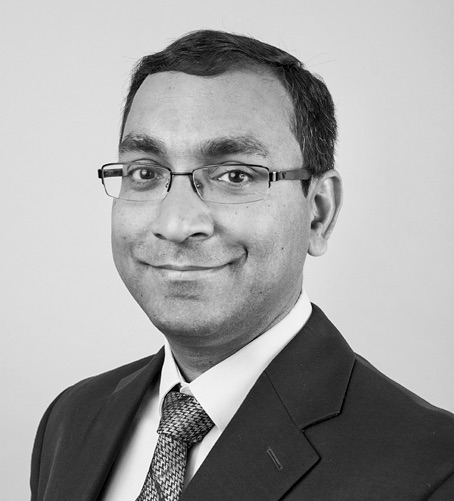}}]{Rahul Jain}
Rahul Jain is the Kenneth C. Dahlberg Early Career Chair in the Electrical Engineering department and an Associate Professor since February 2013 at the University of Southern California, Los Angeles, CA. He received his Ph.D. in EECS, and an M.A. in Statistics, both from the University of California, Berkeley. He also received an M.S. in ECE from Rice University where he was a recipient of the Texas Instruments Fellowship. He received a B.Tech in EE from the Indian Institute of Technology, Kanpur. He has received numerous awards including a James H. Zumberge Faculty Research and Innovation Award in 2009, an IBM Faculty Award, the NSF CAREER Award in 2010 and the ONR Young Investigator Award in 2012. He is a US Fulbright Specialist Scholar for 2017-2020, and was a Fulbright Visiting Professor at the Indian Institute of Science, Bangalore in 2018. His research interests have current focus on AI for Stochastic Systems, Online, Statistical and Reinforcement Learning. He has also worked on Queueing Systems, Power System Economics, Network Economics and Game Theory.

\end{IEEEbiography}

\begin{IEEEbiography}[{\includegraphics[width=1in,height=1.25in,clip,keepaspectratio]{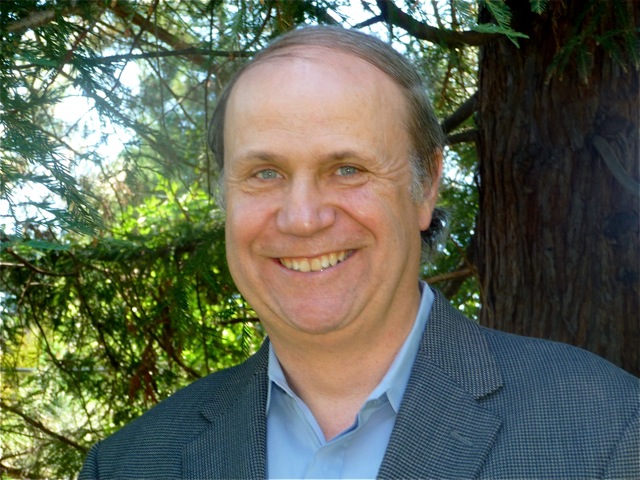}}]{Peter Glynn}
Peter W. Glynn is the Thomas Ford Professor in the Department of Management Science and Engineering (MS\&E) at Stanford University, and also holds a courtesy appointment in the Department of Electrical Engineering. He received his Ph.D in Operations Research from Stanford University in 1982. He then joined the faculty of the University of Wisconsin at Madison, where he held a joint appointment between the Industrial Engineering Department and Mathematics Research Center, and courtesy appointments in Computer Science and Mathematics. In 1987, he returned to Stanford, where he joined the Department of Operations Research. From 1999 to 2005, he served as Deputy Chair of the Department of Management Science and Engineering, and was Director of Stanford's Institute for Computational and Mathematical Engineering from 2006 until 2010. He served as Chair of MS\&E from 2011 through 2015. He is a Fellow of INFORMS and a Fellow of the Institute of Mathematical Statistics, and was an IMS Medallion Lecturer in 1995 and INFORMS Markov Lecturer in 2014. He was co-winner of the Outstanding Publication Awards from the INFORMS Simulation Society in 1993, 2008, and 2016, was a co-winner of the Best (Biannual) Publication Award from the INFORMS Applied Probability Society in 2009, was the co-winner of the John von Neumann Theory Prize from INFORMS in 2010, and will give the INFORMS Philip McCord Morse Lecture in 2020. In 2012, he was elected to the National Academy of Engineering. He was Founding Editor-in-Chief of Stochastic Systems and served as Editor-in-Chief of Journal of Applied Probability and Advances in Applied Probability from 2016 to 2018. His research interests lie in simulation, computational probability, queueing theory, statistical inference for stochastic processes, and stochastic modeling.
\end{IEEEbiography}
\end{document}